\documentclass[12pt]{amsart}
\usepackage[utf8]{inputenc}
\usepackage{amssymb}
\usepackage{cite}
\usepackage{amsthm}
\usepackage{enumerate}
\usepackage{tikz}
\usepackage{caption}
\usepackage{subcaption}
\allowdisplaybreaks
\newcommand{\R}{\mathbb R}
\newcommand{\E}{\mathbb E}

\newcommand{\N}{\mathbb N}

\newtheorem{theorem}{Theorem}
\newtheorem{lemma}[theorem]{Lemma}

\theoremstyle{remark}
\newtheorem{example}{Example}
\newtheorem{remark}{Remark}

\begin{document}

\author[Zs. Páles]{Zsolt P\'ales}
\address{Institute of Mathematics, University of Debrecen, 
Egyetem tér 1, 4032 Debrecen, Hungary}
\email{pales@science.unideb.hu}
\author[T. Szostok]{Tomasz Szostok}
\address{Institute of Mathematics, University of Silesia in Katowice, Bankowa 14, 40-007 Katowice, Poland}
\email{tomasz.szostok@us.edu.pl}
\title[Ordering measures by higher-order monotone functions]{Orderings on measures induced by higher-order monotone functions}
\thanks{{\it Corresponding author}: Tomasz Szostok}
\thanks{The research of the first author was supported by the K-134191 NKFIH Grant.}
\begin{abstract}
The main aim of this paper is to study the functional inequality
\begin{equation*}
\int_{[0,1]}f\bigl((1-t)x+ty\bigr)d\mu(t)\geq 0, \qquad x,y\in I \mbox{ with } x<y,
\end{equation*}
for a continuous unknown function $f:I\to\R$, where $I$ is a nonempty open real interval and $\mu$ is a signed and bounded Borel measure on $[0,1]$. We derive necessary as well as sufficient conditions for its validity in terms of higher-order monotonicity properties of $f$.

Using the results so obtained we can derive sufficient conditions under which the inequality
$$\E f(X)\leq \E f(Y)$$
is satisfied by all functions which are simultaneously: $k_1$-increasing (or decreasing), $k_2$-increasing (or decreasing), \dots , $k_l$-increasing (or decreasing) for given nonnegative integers $k_1,\dots,k_l.$ This extends several well-known results on stochastic ordering.

A necessary condition for the $(n,n+1,\dots,m)$-increasing ordering is also presented.
\end{abstract}

\maketitle
\section{Introduction}
In recent years, stochastic ordering tools were successfully applied to obtain new results in the theory of functional inequalities; see, for example, \cite{Rajba},  \cite{MRW} for the application of the Ohlin lemma and  \cite{Ohlin}, \cite{OSz} for the application of Levin-Stechkin theorem. In addition, the orderings involving higher-order convex functions (see \cite{DLS}) were used among others in \cite{Szres} and \cite{Sznum}. 
See also \cite{CLLM} for recent applications of higher-order convex orderings in economics and \cite{CW} for the application in queueing systems. 

In this paper, we will use the notion of $n$-increasing ($n$-decreasing) function instead of $n$-convex ($n$-concave) since we intend to cover the cases of $f$ being positive (negative) or increasing (decreasing) as well.
To this end, we need to recall the notion of higher-order divided differences. Let $I\subset\R$ be a nonempty open interval throughout this paper. Then, for a function $f\colon I\to\R$, for $n\in\N\cup\{0\}$ and for all pairwise distinct elements $x_0,\dots,x_n\in I$, we define  
$$
f[x_0,\dots,x_n]
:=\sum_{j=0}^n\frac{f(x_j)}{\prod_{k=0,k\neq j}^n(x_j-x_k)},
$$
which is called the \emph{$n$th-order divided difference} for $f$ at $x_0,\dots,x_n$. It is easy to see that the $n$th-order divided difference is a symmetric function of its arguments.
We say that $f$ is \emph{$n$-increasing} if, for all pairwise distinct elements $x_0,\dots,x_n\in I$, we have that $f[x_0,\dots,x_n]\geq 0$. Observe that $0$-increasingness means nonnegativity, $1$-increasingness is equivalent to increasingness, and $2$-increasingness coincides with convexity. Analogously, we say
that $f$ is \emph{$n$-decreasing} if $(-f)$ is $n$-increasing, i.e., for all pairwise distinct elements $x_0,\dots,x_n\in I$, we have that $f[x_0,\dots,x_n]\leq 0$. In this case, $0$-decreasingness means nonpositivity, $1$-decreasingness is equivalent to decreasingness, and $2$-decreasingness coincides with concavity.

In the following theorems, we present the characterizations of $n$-in\-cre\-asing\-ness due to Popoviciu \cite{P}, which may also be found in Kuczma's book \cite{Kuczma} stated as Theorems 15.8.4, 15.8.5, and 15.8.6, respectively.

\begin{theorem}\label{thm:Popoviciu1}
Let $n\in\N\setminus\{1\}$. Then a function $f:I\to\R$ is $n$-increasing if and only if it is $(n-2)$-times continuously differentiable and $f^{(n-2)}$ is convex on $I$.
\end{theorem}

\begin{theorem}\label{thm:Popoviciu2}
Let $n\in\N$ and let $f:I\to\R$ be $(n-1)$-times continuously differentiable. Then $f$ is $n$-increasing if and only if $f^{(n-1)}$ is increasing on $I$.
\end{theorem}

\begin{theorem}\label{thm:Popoviciu3}
Let $n\in\N\cup\{0\}$ and let $f:I\to\R$ be $n$-times continuously differentiable. Then $f$ is $n$-increasing if and only if $f^{(n)}$ is nonnegative on $I$.
\end{theorem}

Furthermore, given $n\in\N\cup\{0\}$, the symbols $M_{n^+}(I)$ and $M_{n^-}(I)$ will represent the set of all  $n$-increasing and $n$-decreasing functions defined on $I$, respectively. 

Let now $X, Y$ be two random variables that take values in $I.$ We say that $X$ is smaller than $Y$ in the $n$-increasing order if, for all $f\in M_{n^+}(I)$,
\begin{equation}
\label{norder}
\E(f(X))\leq \E(f(Y))
\end{equation}
provided that the expectations exist (see \cite{Shaked}). 

In \cite{Shaked} and also in \cite{Rajba}, one can find necessary and sufficient conditions for the $n$-increasing order. 
That is, \eqref{norder} is satisfied if and only if the following two conditions are satisfied.
\begin{equation}
\label{lsn}
\left\{\begin{array}{cc}
    \E X^k=\E Y^k & k\in\{1,\dots,n-1\} \\
     \E(X-t)_+^{n-1}\leq\E(Y-t)_+^{n-1}& t\in I. 
\end{array}\right.    
\end{equation}
Note that here and in the sequel, the symbols $u_-$ and $u_+$ represent the negative and positive part of a real number $u$, respectively, defined as
\[
  u_-:=\max(0,-u)=\frac{|u|-u}{2}\qquad\mbox{and}\qquad u_+:=\max(0,u)=\frac{|u|+u}{2}.
\]

If, however, condition \eqref{lsn} is not fulfilled, the class of functions satisfying the inequality \eqref{norder} is not known. We give a simple example of such a situation.

\begin{example}
For a function $f\colon I\to\R$, consider the functional inequality    
\begin{equation}
\label{1-3}
3f\left(\frac{3x+y}{4}\right)+3f\left(\frac{x+3y}{4}\right)\leq f(x)+4f\left(\frac{x+y}{2}\right)+f(y)
\end{equation}
which is supposed to be valid for all $x,y\in I$ with $x<y$.
At the end of the introduction, we will show that every continuous solution of \eqref{1-3} has to be $2$-increasing, that is, convex. On the other hand, we prove now that convexity is not sufficient for the validity of \eqref{1-3}. Indeed, dividing the inequality by 6 side by side, we can rewrite it in the form \eqref{norder}. Let $x,y\in I$ with $x<y$ be fixed and consider the random variables $X$ and $Y$ defined by
\begin{align*}
 &\mathbb{P}(X=\tfrac{3x+y}{4})
 =\mathbb{P}(X=\tfrac{x+3y}{4})=\frac{1}{2},\\ 
 &\mathbb{P}(Y=x)=\mathbb{P}(Y=y)=\frac{1}{6},\quad
 \mathbb{P}(Y=\tfrac{x+y}{2})=\frac{2}{3}.
\end{align*}
Then, the normalized form of the inequality \eqref{norder} is equivalent to \eqref{1-3}.
The first moments of $X$ and $Y$ are equal to $\frac{x+y}{2}$. Therefore, the first condition in \eqref{lsn} holds for $n=2$. On the other hand, we have 
$$
  \E\Bigl(X-\frac{x+y}{2}\Bigr)_+=\frac18(y-x)>
  \frac1{12}(y-x)=\E\Bigl(Y-\frac{x+y}{2}\Bigr)_+.
$$
This shows that the second condition in \eqref{lsn} is violated. Therefore, we can conclude that not all convex functions satisfy \eqref{1-3}. (In particular, the convex function $u\mapsto \big(u-\frac{x+y}{2}\big)_+$ does not
satisfy \eqref{1-3}.)
\end{example}

This paper aims to provide tools that can be used to deal with inequalities of that kind. The inspiration for our approach may again be found in the monograph
\cite{Shaked}, where it is mentioned (Theorem 4.A.2) that the inequality \eqref{norder} holds for all increasing and convex functions $f$, i.e., for all $f\in M_{1^+}(I)\cap M_{2^+}(I)$, if and only if 
$$\E\bigl((X-t)_+\bigr)\leq \E\bigl((Y-t)_+\bigr)$$
for all $t\in\R.$ 
As we can see, even if an inequality is not satisfied by all functions that are monotone of some order, it may be satisfied if we add additional monotonicity properties of some different orders. 

We will collect all terms to one side of the inequality and consider only a bounded and signed measure $\mu$, defined on the Borel subsets of the interval $[0,1]$. Thus, our main purpose is to study the functional inequality
\begin{equation}
\label{1}
\int_{[0,1]}f\bigl((1-t)x+ty\bigr)d\mu(t)\geq 0, \qquad x,y\in I \mbox{ with } x<y,
\end{equation}
for a continuous unknown function $f:I\to\R$,
and to derive necessary as well as sufficient conditions for its validity in terms of higher-order monotonicity properties of $f$.

In order to formulate our main results, for $k\in\N\cup\{0\}$ and for a given bounded and signed measure $\mu$ on the $\sigma$-algebra of Borel subsets of  $[0,1]$, we introduce the functions 
$\mu_k:\R\to\R$ and 
$\mu^-_k,\mu_k^+:[0,1]\to\R$ by
\begin{align*}
   \mu_k(\tau)&:=\int_{[0,1]}(t-\tau)^kd\mu(t),\\[1mm]
  \mu_k^-(\tau)&:=\int_{[0,\tau)}(t-\tau)^kd\mu(t),\\[1mm]
  \mu_k^+(\tau)&:=\int_{(\tau,1]}(t-\tau)^kd\mu(t).
\end{align*}
Here, in the definition of $\mu_k$, when $k=0$ and $t=\tau$, we adopt the convention $0^0:=1$.
Observe that $\mu_k^-(0)=\mu_k^+(1)=0$. Furthermore, $\mu_k(0)$ equals the $k$th-order moment of the measure $\mu$. A relationship among these functions will be established in Lemma~\ref{Lem:mk} in the next section.

The main result of \cite{BesPal} is recalled in the following theorem.
\begin{theorem}\label{Thm:BP} 
    Let $\mu$ be a nonzero bounded signed Borel measure on $[0,1]$. 
    Assume that $n$ is the smallest non-negative integer such that $\mu_n(0)\neq 0.$ If $f:I\to\R$ is a continuous function satisfying the integral inequality 
	\eqref{1}, then $\mu_n(0)\cdot f$ is $n$-increasing.
\end{theorem}
\begin{remark}
Considering a measure 
$$
\mu=\delta_0-3\delta_{\frac14}+4\delta_{\frac12}-3\delta_{\frac34}+\delta_1,
$$
we may rewrite inequality \eqref{1-3} in the form \eqref{1}. Since we have
$\mu_0(0)=\mu_1(0)=0$ and $\mu_2(0)=\frac{7}{16},$ according to Theorem \ref{Thm:BP}, we find that every continuous solution of \eqref{1-3} is a $2$-increasing (i.e., a convex) function. 
As we remember, not all convex functions satisfy  \eqref{1-3}. Thus, we need to find some additional conditions to the convexity to obtain a class of functions that satisfy \eqref{1-3}. We will return to this inequality in the last part of the paper.   
\end{remark}

\section{Auxiliary results}

The following lemma establishes a recursive formula for the sequences of functions $(\mu_k^-)_{k\geq0}$ and $(\mu_k^+)_{k\geq0}$. In what follows, the characteristic function of a set $A\subseteq\R$ will be denoted by
$\chi_A$.

\begin{lemma}\label{Lem:mk}
Let $\mu$ be a nonzero bounded signed Borel measure on $[0,1]$. Then, for all $k\in\N$ and $\tau\in[0,1]$, we have
\begin{align}\label{Eq:mk-}
    \mu_k^-(\tau)&=-k\int_0^\tau\mu_{k-1}^-(t)dt,\\[1mm]
%    \mu_k(\tau)&=k\int_0^1\mu_{k-1}(t)dt,\label{Eq:mk}\\[2mm]
    \mu_k^+(\tau)&=k\int_\tau^1\mu_{k-1}^+(t)dt.\label{Eq:mk+}
\end{align}
\end{lemma}

\begin{proof} 
We are going to prove the equality \eqref{Eq:mk-} for $k=1$ and for $\tau\in[0,1]$. Applying Fubini's theorem and the Newton--Leibniz formula, we obtain 
\begin{align*}
  \int_0^\tau\mu_0^-(t)dt
  &=\int_0^\tau\bigg(\int_{[0,t)}1d\mu(s)\bigg)dt
  =\int_0^\tau\bigg(\int_{[0,1]}\chi_{[0,t)}(s)d\mu(s)\bigg)dt\\
  &=\int_{[0,1]}\bigg(\int_0^\tau\chi_{[0,t)}(s)dt\bigg)d\mu(s)
  =\int_{[0,1]}\bigg(\int_0^\tau\chi_{(s,1]}(t)dt\bigg)d\mu(s)\\
  &=\int_{[0,1]}(\tau-s)_+\mu(s)
  =\int_{[0,\tau)}(s-\tau)\mu(s)=-\mu_1^-(\tau).
\end{align*}

For $k>1$ and for $\tau\in[0,1]$, using Fubini's theorem again, we get
\begin{align*}
  \int_0^\tau\mu_{k-1}^-(t)dt
  &=\int_0^\tau\bigg(\int_{[0,t)}(s-t)^{k-1} d\mu(s)\bigg)dt
  =\int_0^\tau\bigg(\int_{[0,1]}(-(t-s)_+)^{k-1} d\mu(s)\bigg)dt\\
  &=\int_{[0,1]}\bigg(\int_0^\tau(-(t-s)_+)^{k-1} dt\bigg)d\mu(s)
  =\int_{[0,1]}\bigg[-\frac1k(-(t-s)_+)^k \bigg]_{t=0}^{t=\tau}d\mu(s)\\
  &=\int_{[0,1]}\bigg(-\frac1k(-(\tau-s)_+)^k+\frac1k(-(0-s)_+)^k \bigg)d\mu(s)\\
  &=\int_{[0,1]}-\frac1k(-(\tau-s)_+)^kd\mu(s)
  =-\int_{[0,\tau)}-\frac1k(s-\tau)^kd\mu(s)=-\frac1k\mu_k^+(\tau),
\end{align*}
which completes the proof of the equality \eqref{Eq:mk-} for $k>1$.

To prove the equality \eqref{Eq:mk+} for $k=1$ and for $\tau\in[0,1]$, using Fubini's theorem and the Newton--Leibniz formula, we obtain 
\begin{align*}
  \int_\tau^1\mu_0^+(t)dt
  &=\int_\tau^1\bigg(\int_{(t,1]}1d\mu(s)\bigg)dt
  =\int_\tau^1\bigg(\int_{[0,1]}\chi_{(t,1]}(s)d\mu(s)\bigg)dt\\
  &=\int_{[0,1]}\bigg(\int_\tau^1\chi_{(t,1]}(s)dt\bigg)d\mu(s)
  =\int_{[0,1]}\bigg(\int_\tau^1\chi_{[0,s)}(t)dt\bigg)d\mu(s)\\
  &=\int_{[0,1]}(s-\tau)_+\mu(s)
  =\int_{(\tau,1]}(s-\tau)\mu(s)=\mu_1^+(\tau).
\end{align*}
For $k>1$ and for $\tau\in[0,1]$, using Fubini's theorem again, we get
\begin{align*}
  \int_\tau^1\mu_{k-1}^+(t)dt
  &=\int_\tau^1\bigg(\int_{(t,1]}(s-t)^{k-1} d\mu(s)\bigg)dt
  =\int_\tau^1\bigg(\int_{[0,1]}(s-t)_+^{k-1} d\mu(s)\bigg)dt\\
  &=\int_{[0,1]}\bigg(\int_\tau^1(s-t)_+^{k-1} dt\bigg)d\mu(s)\
  =\int_{[0,1]}\bigg[-\frac1k(s-t)_+^k \bigg]_{t=\tau}^{t=1}d\mu(s)\\
  &=\int_{[0,1]}\bigg(-\frac1k(s-1)_+^k+\frac1k(s-\tau)_+^k \bigg)d\mu(s)
  =\int_{[0,1]}\frac1k(s-\tau)_+^kd\mu(s)=\frac1k\mu_k^+(\tau),
\end{align*}
which completes the proof of the equality \eqref{Eq:mk+} for $k>1$.
\end{proof}

Before presenting sufficient conditions, we will need a lemma connected to a smoothing technique that will be used later. We use here a similar approach to that used in the paper \cite{BesPal}. Namely, let  
$h:\R\to\R$ be a fixed nonnegative function of class $C^\infty$ such that $h(u)=0$ for $u\not\in(-1,1)$ and 
$$\int_{-1}^1h(t)dt=1.$$
Then, for a given $\varepsilon>0$, define the function $h_\varepsilon\colon\R\to\R$ by
$$
h_\varepsilon(t):=\frac{1}{\varepsilon}h\left(\frac{t}{\varepsilon}\right),\qquad t\in\R.
$$
Let $f:I\to\R$ be a given continuous function. Put $I_\varepsilon=(I-\varepsilon)\cap(I+\varepsilon)$ and  define $f_\varepsilon:I_\varepsilon\to\R$ as the convolution of $f$ with $h$ as follows
\begin{equation}
\label{fepsdef}
f_\varepsilon(t):=\int_{-\varepsilon}^\varepsilon f(t-s)h(s)ds,\qquad t\in I_\varepsilon.
\end{equation}
It is well-known (see, for example, \cite[$\mathsection$ 18.14]{Zeidler}) that such function is of the class $C^\infty$ and tends to $f$ as $\varepsilon$ tends to zero (moreover, the convergence is uniform) on any compact subinterval of the interior of $I$. In the following lemma, we will establish that $f_\varepsilon$ also inherits the monotonicity properties of the function $f.$

\begin{lemma}
\label{fepsilon}
Let $f\colon I\to\R$ be a continuous function. If $f$ is $n$-increasing on $I$ for some $n\in\N\cup\{0\}$, then, for all $\varepsilon>0$, the function $f_\varepsilon$ given by \eqref{fepsdef} is $n$-increasing on $I_\varepsilon.$
\end{lemma}
\begin{proof} 
Assume that $f$ is $n$-increasing on $I$ for some $n\in\N\cup\{0\}$. Let $x_0,\dots,x_n$ be pairwise distinct elements of $I_\varepsilon$.
Then, by the $n$-increasingness of $f$, for all $s\in(-\varepsilon,\varepsilon)$, we have that
$$
  f[x_0-s,\dots,x_n-s]\geq0.
$$
Therefore, using the defining formula of $n$th-order divided differences, the definition of $f_\varepsilon$ and the nonnegativity of $h_\varepsilon$, we get
\begin{align*}
f_\varepsilon[x_0,\dots,x_n]
&=\sum_{i=0}^n\frac{f_\varepsilon (x_i)}{\prod_{j=0,j\neq i}^n(x_i-x_j)}\\
&=\sum_{i=0}^n\frac{1}{\prod_{j=0,j\neq i}^n(x_i-x_j)}\int_{-\varepsilon}^\varepsilon f(x_i-s)h_{\varepsilon}(s)ds\\
&=\int_{-\varepsilon}^\varepsilon\sum_{i=0}^n\frac{f(x_i-s)}{\prod_{j=0,j\neq i}^n(x_i-x_j)} h_{\varepsilon}(s)ds\\
&=\int_{-\varepsilon}^\varepsilon\sum_{i=0}^n\frac{f(x_i-s)}{\prod_{j=0,j\neq i}^n\bigl((x_i-s)-(x_j-s)\bigr)} h_{\varepsilon}(s)ds\\
&=\int_{-\varepsilon}^\varepsilon f[x_0-s,\dots,x_n-s] h_{\varepsilon}(s)dv\geq 0.
\end{align*}
This completes the proof of the $n$-increasingness of $f_\varepsilon$ on $I_\varepsilon$.
\end{proof}

\section{Sufficient conditions for the inequality \eqref{1}}

The next theorem presents one of the main results of this paper. 

\begin{theorem}\label{sufth}
Let $n\in\N$, let $\mu$ be a bounded signed Borel measure on $[0,1]$, and let $\lambda\in [0,1]$.
Assume that $f\colon I\to\R$ is a continuous function possessing the following properties:
\begin{enumerate}[(a)]
\item For all $k\in\{0,\dots,n-1\}$, the function 
$\mu_k(\lambda)\cdot f$ is $k$-increasing. 
\item We have one of the following two possibilities: 
\begin{enumerate}[(i)]
\item Either $f$ is $n$-increasing and the following inequalities hold:
\[
\qquad\qquad\mu_{n-1}^-(\tau)\leq0\mbox{ for }\tau\in[0,\lambda)
\quad\mbox{and}\quad
\mu_{n-1}^+(\tau)\geq0\mbox{ for }\tau\in(\lambda,1],
\]
\item or $f$ is $n$-decreasing and the following inequalities hold:
\[
\qquad\qquad\mu_{n-1}^-(\tau)\geq0\mbox{ for }\tau\in[0,\lambda)
\quad\mbox{and}\quad
\mu_{n-1}^+(\tau)\leq0\mbox{ for }\tau\in(\lambda,1].
\]
\end{enumerate}
\end{enumerate}
Then $f$ satisfies the inequality \eqref{1}.
\end{theorem}

\begin{proof}
Take $x<y$ from the interior of $I$. First we assume that  $f\in C^\infty([x,y]).$
Then, using Taylor's theorem at the point $p:=(1-\lambda)x+\lambda y$ with the integral remainder term, for every $u\in[x,y]$, we have
$$
f(u)=\sum_{k=0}^{n-1}\frac{f^{(k)}(p)}{k!}(u-p)^k+\int_{p}^u \frac{f^{(n)}(s)}{(n-1)!}(u-s)^{n-1}ds.
$$
Therefore, substituting $u:=(1-t)x+ty$ into the above equality and then integrating it with respect to $\mu$, we get
\begin{align}
\label{intf}
 &\int_{[0,1]}f\bigl((1-t)x+ty\bigr)d\mu(t)\\
 &=\sum_{k=0}^{n-1}\frac{f^{(k)}(p)}{k!}\int_{[0,1]}\bigl((1-t)x+ty-p\bigr)^k d\mu(t)\nonumber\\
 &\quad +\int_{[0,1]}\left(\int_p^{(1-t)x+ty}\frac{f^{(n)}(s)}{(n-1)!}\bigl((1-t)x+ty-s\bigr)^{n-1}ds\right)d\mu(t).\nonumber
\end{align}
Let $k\in\{0,\dots,n-1\}$ be fixed. Then
\begin{align*}
\int_{[0,1]}\bigl((1-t)x&+ty-p\bigr)^kd\mu(t) \\
&=\int_{[0,1]}\bigl((1-t)x+ty-((1-\lambda)x+\lambda y)\bigr)^kd\mu(t)\\
&=(y-x)^k\int_{[0,1]}(t-\lambda)^kd\mu(t)=(y-x)^k\cdot\mu_k(\lambda).
\end{align*}
Combining condition \textit{(a)} with Theorem \ref{thm:Popoviciu3}, we obtain that 
$$
\mu_k(\lambda)\cdot f^{(k)}(p)\geq 0.
$$
Thus, we conclude that
\begin{align}\label{sum}
\sum_{k=0}^{n-1}\frac{f^{(k)}(p)}{k!}&\int_{[0,1]}\bigl((1-t)x+ty-p\bigr)^kd\mu(t)\\
&=\sum_{k=0}^{n-1}\frac{f^{(k)}(p)}{k!}(y-x)^k\cdot\mu_k(\lambda)\geq0.\nonumber
\end{align}

To obtain the nonnegativity of the second term on the right-hand side of \eqref{intf}, we will use condition \textit{(b)}. Without loss of generality, we assume that condition \textit{(i)} holds. Then, by Theorem~\ref{thm:Popoviciu3}, we have that $f^{(n)}$ is nonnegative on $I$. We split the second term on the right-hand side of \eqref{intf} as follows:
\begin{align}\label{split}
\int_{[0,1]}&\left(\int_p^{(1-t)x+ty}\frac{f^{(n)}(s)}{(n-1)!}\bigl((1-t)x+ty-s\bigr)^{n-1}ds\right)d\mu(t)\\
&=\int_{[0,\lambda]}\left(\int_{p}^{(1-t)x+ty}\frac{f^{(n)}(s)}{(n-1)!}\bigl((1-t)x+ty-s\bigr)^{n-1}ds\right)d\mu(t)\nonumber\\
&\quad+\int_{(\lambda,1]}\left(\int_{p}^{(1-t)x+ty}\frac{f^{(n)}(s)}{(n-1)!}\bigl((1-t)x+ty-s\bigr)^{n-1}ds\right)d\mu(t).\nonumber
\end{align}
Now, we compute the two integrals on the right-hand side separately.
For $t\in[0,\lambda]$, we have that $x\leq (1-t)x+ty\leq (1-\lambda)x+\lambda y=p$, therefore, by using Fubini's theorem, we get
\begin{align*}
\int_{[0,\lambda]}&\bigg(\int_{p}^{(1-t)x+ty}\frac{f^{(n)}(s)}{(n-1)!}\bigl((1-t)x+ty-s\bigr)^{n-1}ds\bigg)d\mu(t)\\
&=\int_{[0,\lambda]}\left(-\int_{(1-t)x+ty}^{p}\frac{f^{(n)}(s)}{(n-1)!}\bigl((1-t)x+ty-s\bigr)^{n-1}ds\right)d\mu(t)\\
&=-\int_{[0,\lambda]}\left(\int_x^p\chi_{((1-t)x+ty,p]}(s) \cdot \frac{f^{(n)}(s)}{(n-1)!}\bigl((1-t)x+ty-s\bigr)^{n-1}ds\right)d\mu(t)\\
&=-\int_x^p\left(\int_{[0,\lambda]}\chi_{((1-t)x+ty,p]}(s) \cdot \frac{f^{(n)}(s)}{(n-1)!}\bigl((1-t)x+ty-s\bigr)^{n-1}d\mu(t)\right)ds\\
&=-\int_x^p\bigg(\int_{[0,\frac{s-x}{y-x})}\frac{f^{(n)}(s)}{(n-1)!}\bigl((1-t)x+ty-s\bigr)^{n-1}d\mu(t)\bigg)ds\\
&=\int_x^p\frac{f^{(n)}(s)}{(n-1)!}(y-x)^{n-1}\bigg(-\int_{\left[0,\frac{s-x}{y-x}\right)}\left(t-\frac{s-x}{y-x}\right)^{n-1}d\mu(t)\bigg)ds\\
&=\int_x^p\frac{f^{(n)}(s)}{(n-1)!}(y-x)^{n-1}\cdot(-1)\cdot\mu_{n-1}^-\left(\frac{s-x}{y-x}\right)ds\geq0.
\end{align*}
To see that the last inequality is valid, observe that, for $s\in[x,p)$, the ratio $\tau:=\frac{s-x}{y-x}$ belongs to $[0,\lambda)$ and hence, by condition \textit{(i)}, we have that $\mu_{n-1}^-(\tau)\leq0$.

For $t\in (\lambda,1]$, we have that $p=(1-\lambda)x+\lambda y<(1-t)x+ty\leq y$; therefore, by using Fubini's theorem again, we obtain
\begin{align*}
\int_{(\lambda,1]}&\bigg(\int_{p}^{(1-t)x+ty}\frac{f^{(n)}(s)}{(n-1)!}\bigl((1-t)x+ty-s\bigr)^{n-1}ds\bigg)d\mu(t)\\
&=\int_{(\lambda,1]}\left(\int_p^{y}\chi_{[p,(1-t)x+ty)}(s) \cdot \frac{f^{(n)}(s)}{(n-1)!}\bigl((1-t)x+ty-s\bigr)^{n-1}ds\right)d\mu(t)\\
&=\int_p^{y}\left(\int_{(\lambda,1]}\chi_{[p,(1-t)x+ty)}(s) \cdot \frac{f^{(n)}(s)}{(n-1)!}\bigl((1-t)x+ty-s\bigr)^{n-1}d\mu(t)\right)ds\\
&=\int_p^{y}\bigg(\int_{\left(\frac{s-x}{y-x},1\right]}\frac{f^{(n)}(s)}{(n-1)!}\bigl((1-t)x+ty-s\bigr)^{n-1}d\mu(t)\bigg)ds\\
&=\int_p^{y}\frac{f^{(n)}(s)}{(n-1)!}(y-x)^{n-1}\bigg(\int_{\left(\frac{s-x}{y-x},1\right]}\left(t-\frac{s-x}{y-x}\right)^{n-1}d\mu(t)\bigg)ds\\
&=\int_p^{y}\frac{f^{(n)}(s)}{(n-1)!}(y-x)^{n-1}\cdot\mu_{n-1}^+\left(\frac{s-x}{y-x}\right)ds\geq0.
\end{align*}
Observe that, for $s\in(p,y]$, the ratio $\tau:=\frac{s-x}{y-x}$ belongs to $(\lambda,1]$ and hence, by condition \textit{(i)}, we have that $\mu_{n-1}^+(\tau)\geq0$. This validates the last inequality above

Given the two inequalities obtained and \eqref{split}, it follows that
\[
\int_{[0,1]}\bigg(\int_p^{(1-t)x+ty}\frac{f^{(n)}(s)}{(n-1)!}\bigl((1-t)x+ty-s\bigr)^{n-1}ds\bigg)d\mu(t)\geq0.
\]
This inequality, together with \eqref{sum} used in \eqref{intf}, yields that inequality \eqref{1} is valid.

To finish the proof, we need to consider the case when $f$ is not necessarily of the class $C^\infty.$ First chose $\varepsilon_0>0$ so that $[x,y]\subseteq I_{\varepsilon_0}$. Then, for all $\varepsilon\in(0,\varepsilon_0)$, we have that $[x,y]\subseteq I_\varepsilon$. On the other hand, according to Lemma \ref{fepsilon}, conditions \textit{(a)} and \textit{(b)} imply that $\mu_k(\lambda)\cdot f_\varepsilon$ is $k$-increasing for all $k\in\{0,\dots,n-1\}$, and $f_\varepsilon$  (or $-f_\varepsilon$) is $n$-increasing on $I_\varepsilon$. Since $f_\varepsilon$ is infinitely many times differentiable on $I_\varepsilon$, it is also infinitely many times differentiable on $[x,y]$ if $\varepsilon\in(0,\varepsilon_0)$. From the first part of the proof, it follows that
$$
\int_{[0,1]}f_\varepsilon\bigl((1-t)x+ty\bigr)d\mu(t)\geq 0
$$
holds for all $\varepsilon\in(0,\varepsilon_0)$.
Now, using that $f_\varepsilon$ uniformly converges to $f$ on $[x,y]$ as $\varepsilon\to0$, upon taking the limit, we get that \eqref{1} holds for $x<y$ belonging to the interior of $I$. Using the continuity of $f$ at the endpoints of $I$ if necessary, we can see that the inequality is also valid for all $x<y$ from $I$.
\end{proof}

\section{Necessary conditions for the inequality \eqref{1}}

To derive the necessary conditions for the inequality \eqref{1}, we consider the following families of the test functions. 

For fixed $p\in\R$, $n\in\N$, and $u\in I$, define
\[
  \varphi_{p,n}(u):=(u-p)^n, \qquad
  \varphi^+_{p,n}(u):=(u-p)^n_+,\qquad
  \varphi^-_{p,n}(u):=(u-p)^n_-
\]
and, additionally, let 
\[
  \varphi_{p,0}:=1, \qquad
  \varphi^+_{p,0}:=\chi_{(p,\infty)\cap I}\qquad
  \varphi^-_{p,0}:=\chi_{(-\infty,p)\cap I}.
\]

\begin{lemma}\label{lem:pn}Let $p\in\R$ and $n,k\in\N\cup\{0\}$. The following assertions hold:
\begin{enumerate}[(i)]
\item If $0\leq k<n$ and $p\leq\inf I$, then 
$\varphi_{p,n}\in M_{k^+}(I)\setminus M_{k^-}(I)$.
\item If $0\leq k<n$, $n-k$ is even and $p\in I$, then $\varphi_{p,n}\in M_{k^+}(I)\setminus M_{k^-}(I)$.
\item If $0\leq k<n$, $n-k$ is odd and $p\in I$, then $\varphi_{p,n}\not\in M_{k^+}(I)\cup M_{k^-}(I)$.
\item If $0\leq k<n$ and $\sup I\leq p$, then 
$(-1)^{n-k}\varphi_{p,n}\in M_{k^+}(I)\setminus M_{k^-}(I)$.
\item If $k=n$, then $\varphi_{p,n}\in M_{k^+}(I)\setminus M_{k^-}(I)$. 
\item If $k>n$, then $\varphi_{p,n}\in M_{k^+}(I)\cap M_{k^-}(I)$.
\end{enumerate}
\end{lemma}

\begin{proof} The function $\varphi_{p,n}$ is infinitely many times differentiable; therefore, for the proof of the assertion, we can use Theorem \ref{thm:Popoviciu3}. For all $u\in I$, we have that
\[
  \varphi_{p,n}^{(\ell)}(u)
  =\begin{cases}
  \frac{n!}{(n-\ell)!}(u-p)^{n-\ell} & \mbox{if } 0\leq \ell\leq n,\\
  0 & \mbox{if } \ell>n.\\
  \end{cases}
\]
By analyzing the positivity and the negativity of the $k$th order derivative over $I$, the assertions follow straightforwardly.
\end{proof}

\begin{lemma}\label{lem:pn+} Let $p\in I$ and $n,k\in\N\cup\{0\}$. The following assertions hold:
\begin{enumerate}[(i)]
\item If $0\leq k\leq n+1$, then $\varphi^+_{p,n}\in M_{k^+}(I)\setminus M_{k^-}(I)$.
\item If $k>n+1$, then $\varphi^+_{p,n}\not\in M_{k^+}(I)\cup M_{k^-}(I)$.
\end{enumerate}
\end{lemma}

\begin{proof} 
If $n=0$, then $\varphi_{p,0}^+$ is a nonnegative (but not nonpositive) and increasing (but not decreasing) function which is discontinuous at $p$. Therefore, the assertion $(i)$ holds for $k\in\{0,1\}$. Observing that $\varphi_{p,0}^+$ is discontinuous at $p$, by Theorem \ref{thm:Popoviciu1}, we get that it cannot be $k$-monotone if $k>1$, which shows that assertion $(ii)$ is also valid in this case. Thus, from now on, we can assume that $n\geq1$.

The function $\varphi^+_{p,n}$ is infinitely many times differentiable on $\R$ except at $u=p$. At this point, it is differentiable only $n-1$ times. In addition,
for all $u\in I$, we have that
\[
  \big(\varphi_{p,n}^+\big)^{(\ell)}(u)
  =\frac{n!}{(n-\ell)!}(u-p)_+^{n-\ell} \qquad \mbox{if}\quad 0\leq \ell\leq n-1.
\]
First, we verify assertion (i). 

For $k=0$, the assertion in $(i)$ is obvious.
If $1\leq k\leq n$, then $0\leq k-1\leq n-1$, and hence the $(k-1)$st-order derivative of $\varphi_{p,n}^+$ exists, and it is an increasing (but not decreasing) continuous function. Therefore, according to Theorem \ref{thm:Popoviciu2}, the function $\varphi_{p,n}^+$ is $k$-increasing but not $k$-decreasing on $I$. 

If $k=n+1$, then $\varphi_{p,n}^+$ is $k-2=n-1$ times differentiable and its $(n-1)$st-order derivative equals $n!\varphi_{p,1}^+$, which is a convex (but not concave) function. In view of Theorem \ref{thm:Popoviciu2}, it follows that $\varphi_{p,n}^+$ is $(n+1)$-increasing but not $(n+1)$-decreasing on $I$.

Finally, if $k>n+1$, i.e., $k\geq n+2$ and if $\varphi^+_{p,n}\in M_{k^+}(I)\cup M_{k^-}(I)$ was valid, then, by Theorem \ref{thm:Popoviciu1}, we get that $f$ is $k-2\geq n$ times continuously differentiable on $I$, which is a contradiction. This completes the proof of assertion (ii).
\end{proof}

\begin{lemma}\label{lem:pn-} Let $p\in I$ and $n,k\in\N\cup\{0\}$. The following assertions hold:
\begin{enumerate}[(i)]
\item If $0\leq k\leq n+1$, then $(-1)^k\varphi^-_{p,n}\in M_{k^+}(I)\setminus M_{k^-}(I)$.
\item If $k>n+1$, then $\varphi^-_{p,n}\not\in M_{k^+}(I)\cup M_{k^-}(I)$.
\end{enumerate}
\end{lemma}

\begin{proof}
The proof is entirely analogous to that of Lemma \ref{lem:pn+}; therefore, it is omitted.
\end{proof}

\begin{lemma} \label{lem:pn0}
Let $p\in\R$ and $n\in\N\cup\{0\}$. Then $f:=\varphi_{p,n}$ satisfies inequality \eqref{1} if and only if 
\[
  \mu_n(\lambda)\geq0 \qquad (\lambda\in J(p)),
\]
where
\[
  J(p)
  :=\begin{cases}
  \bigg(-\infty,\dfrac{p-\inf I}{\sup I-\inf I}\bigg)
   \subseteq(-\infty,0)& \mbox{if } p\leq\inf I,\\[4mm]
  \R & \mbox{if } p\in I,\\[1mm]
  \bigg(1+\dfrac{p-\sup I}{\sup I-\inf I},\infty\bigg)\subseteq(1,\infty)
   & \mbox{if } p\geq\sup I.
  \end{cases}
\]
Analogously, $f:=-\varphi_{p,n}$ satisfies inequality \eqref{1} if and only if 
\[
  \mu_n(\lambda)\leq0 \qquad (\lambda\in J(p)).
\]
\end{lemma}

\begin{proof} 
The inequality \eqref{1} with $f:=\varphi_{p,n}$ reads as follows: for all $x,y\in I$ with $x<y$,
\begin{align*}  
  0&\leq\int_{[0,1]} \bigl((1-t)x+ty-p\bigr)^n d\mu(t)\\
    &=(y-x)^n\int_{[0,1]} (t-\tfrac{p-x}{y-x})^n d\mu(t)
    =(y-x)^n\mu_n(\tfrac{p-x}{y-x}).
\end{align*}
Therefore, validity of inequality \eqref{1} with $f=\varphi_{p,n}$ is equivalent to the property that
\[
  \mu_n(\lambda)\geq0
\]
for all $\lambda$ belonging to the set
\[
  S(p):=\bigg\{\frac{p-x}{y-x}\colon
  x,y\in I,\,x<y\bigg\}.
\]
An elementary computation shows that $S(p)$ coincides with the open interval $J(p)$ defined in the assertion, completing the proof.
\end{proof}

\begin{lemma} \label{lem:pn++}
Let $p\in I$ and $n\in\N\cup\{0\}$. Then $f:=\varphi_{p,n}^+$ satisfies inequality \eqref{1} if and only if 
\begin{equation}\label{pn++}
  \mu_n^+(\lambda)\geq0 \quad (\lambda\in[0,1])
  \qquad\mbox{and}\qquad
  \mu_n(\lambda)\geq0 \quad (\lambda\in(-\infty,0]). 
\end{equation}
\end{lemma}

\begin{proof}
The inequality \eqref{1} with $f:=\varphi_{p,n}^+$ reads as follows: for all $x,y\in I$ with $x<y$,
\[
  0\leq\int_{[0,1]} \varphi_{p,n}^+\bigl((1-t)x+ty\bigr) d\mu(t).
\] 
Observe that, according to its definition, $\varphi_{p,n}^+(u)=0$ if $u\leq p$ and $\varphi_{p,n}^+(u)>0$ if $p<u$. 

If $x<y\leq p$, then $(1-t)x+ty\leq p$, therefore $\varphi_{p,n}^+\bigl((1-t)x+ty\bigr)=0$ for all $t\in[0,1]$. Thus, the above inequality is trivially valid.

If $x\leq p<y$, then $(1-t)x+ty>p$, i.e., $\varphi_{p,n}^+\bigl((1-t)x+ty\bigr)>0$ holds if and only if $t\in (\frac{p-x}{y-x},1]$. Therefore,
\begin{align*}  
  0\leq \int_{[0,1]} \varphi_{p,n}^+\bigl((1-t)x+ty\bigr) d\mu(t)
  &=\int_{(\frac{p-x}{y-x},1]} \bigl((1-t)x+ty-p\bigr)^n d\mu(t)\\
  &=(y-x)^n\int_{(\frac{p-x}{y-x},1]} (t-\tfrac{p-x}{y-x})^n d\mu(t)\\
  &=(y-x)^n\mu_n^+(\tfrac{p-x}{y-x}).
\end{align*}
Now, we can conclude that \eqref{1} holds with $f:=\varphi_{p,n}^+$ for all $x,y\in I$ satisfying $x\leq p<y$ if and only if the first inequality in \eqref{pn++} is valid.

If $p<x<y$, then $(1-t)x+ty>p$, i.e., $\varphi_{p,n}^+\bigl((1-t)x+ty\bigr)>0$ holds for all $t\in [0,1]$. Therefore,
\begin{align*}  
  0\leq \int_{[0,1]} \varphi_{p,n}^+\bigl((1-t)x+ty\bigr) d\mu(t)
  &=\int_{[0,1]} \bigl((1-t)x+ty-p\bigr)^n d\mu(t)\\
  &=(y-x)^n\int_{[0,1]} (t-\tfrac{p-x}{y-x})^n d\mu(t)
  =(y-x)^n\mu_n(\tfrac{p-x}{y-x}).
\end{align*}
Thus, the inequality \eqref{1} holds for $f:=\varphi_{p,n}^+$ and $p<x<y$ if and only if $\mu_n(\lambda)\geq0$ for all $\lambda$ belonging to the set
\[
  S^+(p):=\bigg\{\frac{p-x}{y-x}\colon
  x,y\in I,\,p<x<y\bigg\}=(-\infty,0).
\]
Since $\mu_n(\lambda)$ is a polynomial of $\lambda$, therefore, $\mu_n$ is continuous and, consequently, $\mu_n$ is nonnegative on $(-\infty,0)$ if and only it is nonnegative on $(-\infty,0]$.
This proves that the second inequality in \eqref{pn++} is also a necessary and sufficient condition for the inequality \eqref{1} to be valid with $f:=\varphi_{p,n}^+$ and for all $x,y\in I$ satisfying $p<x<y$.
\end{proof}

\begin{lemma} \label{lem:pn--}
Let $p\in I$ and $n\in\N\cup\{0\}$. Then $f:=\varphi_{p,n}^-$ satisfies inequality \eqref{1} if and only if 
\[
  (-1)^n\mu_n^-(\lambda)\geq0 \quad (\lambda\in[0,1])
  \qquad\mbox{and}\qquad
  (-1)^n\mu_n(\lambda)\geq0 \quad (\lambda\in[1,\infty)).
\]
\end{lemma}

The proof of this lemma is completely analogous to that of Lemma \ref{lem:pn++}; therefore, it is omitted.

Now, we give a necessary condition for the $\bigl(m^+,(m+1)^+,\dots,n^+\bigr)$- mo\-no\-tone ordering. 

\begin{theorem}
\label{necth}
Let $0\leq m<n$ be positive integers, and $\mu$ be a bounded signed measure on $[0,1].$ Assume that the interval $I\subset\R$ is bounded from below. Then the inequality 
\eqref{1} holds for all $$f\in M_{m^+,(m+1)^+,\dots,n^+}(I)$$ if and only if the following conditions are satisfied:
\begin{enumerate}[(i)]
    \item $\mu_0(0)=\cdots=\mu_{m-1}(0)=0$,
    \item $\mu_k(0)\geq0$ holds for all $k\in\{m,\dots,n-1\}$,
\item $\mu_{n-1}^+(\tau)\geq 0$ for all $\tau\in(0,1].$
\end{enumerate}
\end{theorem}

\begin{proof}
To prove the sufficiency part of this theorem, assume that $f\in M_{m^+,(m+1)^+,\dots,n^+}(I)$ and the conditions $(i)$--$(iii)$ are satisfied. Then, in view of conditions $(i)$ and $(ii)$, we trivially have that $\mu_k(0)f$ is $k$-increasing for $k\in\{0,\dots,n-1\}$. Using also condition $(iii)$, we can see that assumptions $(a)$ and $(b)(i)$ of Theorem \ref{sufth} with $\lambda=0$ (even without using that $I$ is bounded from below) are satisfied. Thus, the conclusion of this theorem is valid; hence, $f$ fulfills the inequality \eqref{1}.

In what follows, we verify the necessity of conditions $(i)$--$(iii)$. Assume that the inequality \eqref{1} holds for all $f\in M_{m^+,(m+1)^+,\dots,n^+}(I)$.

To see that condition $(i)$ is necessary, let $k\in\{0,\dots,m-1\}$ and let $p\in I$ be fixed. Then according to assertion $(vi)$ of Lemma \ref{lem:pn}, for all $\ell\in\{m,\dots,n-1\}$, we have that $\varphi_{p,k}\in M_{\ell^+}(I)\cap M_{\ell^-}(I)$, and hence $\pm\varphi_{p,k}\in M_{m^+,(m+1)^+,\dots,n^+}(I)$ holds. Thus, by our assumption, these two functions satisfy \eqref{1}. Applying Lemma \ref{lem:pn0}, it follows that $\mu_k(\lambda)$ and $-\mu_k(\lambda)$ are nonnegative for all $\lambda\in\R$. In particular, for $\lambda=0$, which shows that condition $(i)$ has to be valid.

For the proof of the necessity of condition $(ii)$, take $p:=\inf I>-\infty$ and let $k\in\{m,\dots,n-1\}$.
Then, by assertions $(i)$, $(v)$ and $(vi)$ of Lemma \ref{lem:pn}, the function $f:=\varphi_{p,k}\in M_{\ell^+}(I)$ for all $\ell\in\N\cup\{0\}$ and hence $f\in M_{m^+,(m+1)^+,\dots,n^+}(I)$ holds. Thus, by our assumption, this function satisfies \eqref{1}. Using Lemma \ref{lem:pn0}, it follows that $\mu_k$ is nonnegative on $J(p)=(-\infty,0)$. By the continuity of $\mu_k$, it follows that $\mu_k(0)\geq 0$, which validates condition $(ii)$.

Finally, we proceed to prove the necessity of $(iii)$. Fix $p\in I$. Then, by Lemma \ref{lem:pn+}, the function $f:=\varphi_{p,n-1}^+$ belongs to $M_{0^+,\dots,n^+}(I).$ Thus, by our assumption, it should satisfy \eqref{1}.
Using now Lemma \ref{lem:pn++}, 
with this function $f$, we can obtain that $\mu_{n-1}^+$
is nonnegative on $[0,1]$, that is, condition $(iii)$ must hold.
\end{proof}

\begin{remark}
Observe that we used the assumption $\inf I>-\infty$ only in the proof of the necessity of condition $(ii).$ Thus, the other two conditions are necessary without this assumption. It is an open question whether condition $(ii)$ of Theorem \ref{necth} is necessary if the interval $I$ is not bounded below.  
\end{remark}

A counterpart of Theorem \ref{necth} is contained in the following result.

\begin{theorem}
\label{necthub}
Let $0\leq m<n$ be positive integers, and $\mu$ be a bounded signed measure on $[0,1].$ Assume that the interval $I\subset\R$ is bounded from above. Then the inequality \eqref{1} holds for all 
$$f\in \bigcap_{k=m}^n (-1)^{n-k}M_{k^+}(I)$$ 
if and only if the following conditions are satisfied:
\begin{enumerate}[(i)]
\item $\mu_0(1)=\cdots=\mu_{m-1}(1)=0$,
\item $(-1)^{n-k}\mu_k(1)\geq0$ holds for all $k\in\{m,\dots,n-1\}$,
\item $\mu_{n-1}^-(\tau)\leq 0$ for all $\tau\in[0,1).$
\end{enumerate}
\end{theorem}

\begin{proof}
To prove the sufficiency part of this theorem, assume that $f\in \bigcap_{k=m}^n (-1)^{n-k}M_{k^+}(I)$ and the conditions $(i)$--$(iii)$ are satisfied. Then, by conditions $(i)$ and $(ii)$, we can see that $\mu_k(1)\cdot f$ is $k$-increasing for $k\in\{0,\dots,n-1\}$. Using also condition $(iii)$, it follows that assumptions $(a)$ and $(b)(i)$ of Theorem \ref{sufth} are satisfied with $\lambda=1$. Thus, the conclusion of this theorem holds, and hence $f$ fulfills \eqref{1}.

Next, we verify the necessity of conditions $(i)$--$(iii)$. Assume that the inequality \eqref{1} holds for all $f\in\bigcap_{k=m}^n (-1)^{n-k}M_{k^+}(I)$.

To prove the necessity of condition $(i)$, let $k\in\{0,\dots,m-1\}$ and let $p\in I$ be fixed. Then according to assertion $(vi)$ of Lemma \ref{lem:pn}, for all $\ell\in\{m,\dots,n-1\}$, we have that $\varphi_{p,k}\in M_{\ell^+}(I)\cap M_{\ell^-}(I)$, and hence $\pm\varphi_{p,k}\in \bigcap_{\ell=m}^n (-1)^{n-\ell}M_{\ell^+}(I)$ holds. Thus, by our assumption, these two functions satisfy \eqref{1}. Applying Lemma \ref{lem:pn0}, it follows that $\mu_k(\lambda)$ and $-\mu_k(\lambda)$ are nonnegative for all $\lambda\in\R$. In particular, for $\lambda=1$, which shows that condition $(i)$ must hold.

For the proof of the necessity of condition $(ii)$, take $p:=\sup I<\infty$ and let $k\in\{m,\dots,n-1\}$.
Then, by assertions $(iv)$, $(v)$ and $(vi)$ of Lemma \ref{lem:pn}, we have $f:=(-1)^{n-k}\varphi_{p,k}\in (-1)^{n-\ell}M_{\ell^+}(I)$ for all $\ell\in\N\cup\{0\}$ and hence $f\in \bigcap_{\ell=m}^n (-1)^{n-\ell}M_{\ell^+}(I)$ holds. Thus, by our assumption, this function satisfies \eqref{1}. Using Lemma \ref{lem:pn0}, it follows that $(-1)^{n-k}\mu_k$ is nonnegative on $J(p)=(1,\infty)$. By the continuity of $\mu_k$, it follows that $(-1)^{n-k}\mu_k(1)\geq 0$, which shows that condition $(ii)$ is also necessary.

In the last step, we verify the necessity of $(iii)$. Fix $p\in I$. Then, by Lemma \ref{lem:pn-},  $f:=(-1)^{n}\varphi_{p,n-1}^-$ belongs to $\bigcap_{\ell=0}^n (-1)^{n-\ell}M_{\ell^+}(I).$ Thus, by our assumption, it fulfills \eqref{1}. Using now Lemma \ref{lem:pn--}, 
with this function $f$, we can obtain that $(-1)^{n}(-1)^{n-1}\mu_{n-1}^-=-\mu_{n-1}^-$ is nonnegative on $[0,1]$, that is, condition $(iii)$ must hold.
\end{proof}

\section{Examples}

\begin{example}
For a function $f\colon I\to\R$, consider the functional inequality 
\begin{equation}\label{1-1}
f\left(\frac{x+y}{2}\right)\leq f(y),
\end{equation}
which is supposed to be valid for all $x,y\in I$ with $x<y$.
Let the measure $\mu$ be given by
$$\mu=-\delta_{\frac12}+\delta_1.$$
Then, \eqref{1-1} is equivalent to \eqref{1}. We have that $\mu_0(0)=0$ and $\mu_1(0)=\frac{1}{2}>0$. Therefore, according to Theorem \ref{Thm:BP}, the continuous solutions of \eqref{1-1} have to be increasing. For $\tau\in(0,1]$, we have that
\begin{align*}
\mu_0^+(\tau)
=\int_{(\tau,1]} (t-\tau)^0d\mu(t)
=\mu((\tau,1])
=\begin{cases}
0 & \mbox{if } \tau\in(0,\tfrac{1}{2}),\\
1 & \mbox{if } \tau\in[\tfrac{1}{2},1],
\end{cases}
\end{align*}
which shows that $\mu_0^+$ is nonnegative on $(0,1]$.
Thus, applying Theorem \ref{sufth} with $\lambda=0$ and $n=1$, it follows that every $1$-increasing function is a solution to \eqref{1-1}. 
\end{example}

The next example is a continuation of Example 1 from the introduction.
\begin{example}
\label{ex1-34-31}
Let the measure $\mu$ be given by
$$\mu=\delta_0-3\delta_{\frac14}+4\delta_\frac12-3\delta_\frac34+\delta_1.$$
We will show that the inequality \eqref{1-3} is satisfied by all $f\in M_{2^+,3^+}(I).$
Indeed, we have:
$\mu_0(0)=\mu_1(0)=0$ and $\mu_2(0)=\frac{1}{8},$
i.e., in view of Theorem \ref{Thm:BP} the set of all solutions of \eqref{1-3} is contained in $M_{2^+}(I).$

Now, we use Theorem \ref{sufth} with $\lambda=0.$
The function $\mu_1^+$ changes its sign in the interval $[0,1]$ (see Figure 1 (A)). 
\begin{figure}
     \centering
          \begin{subfigure}[b]{0.3\textwidth}
         \centering        \includegraphics[width=\textwidth, height=0.6\textwidth]{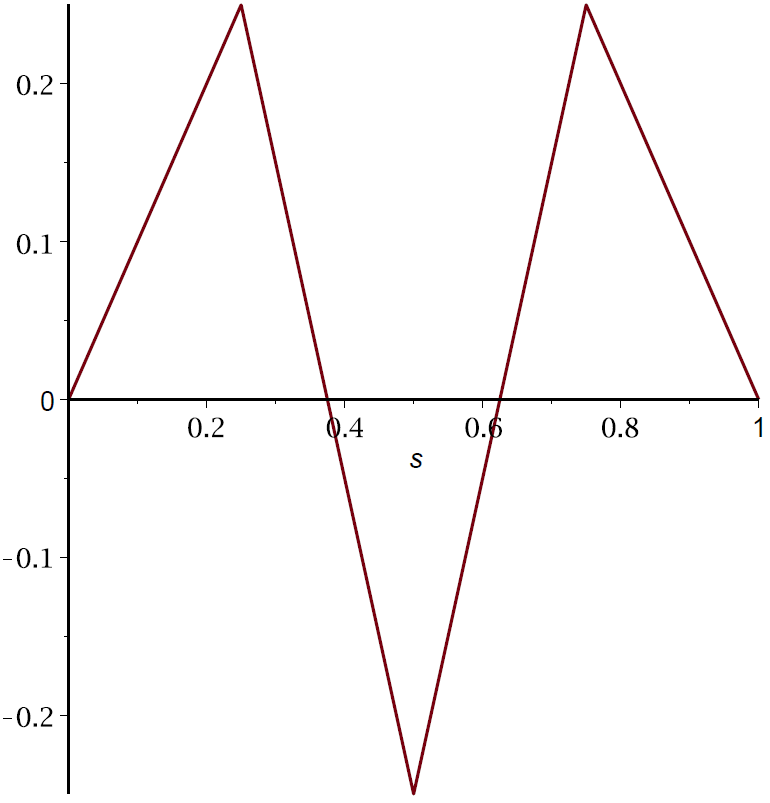}
         \caption{$\mu_1^+$}
             \end{subfigure}
     \hfill
       \begin{subfigure}[b]{0.3\textwidth}
         \centering
\includegraphics[width=\textwidth, height=0.6\textwidth]{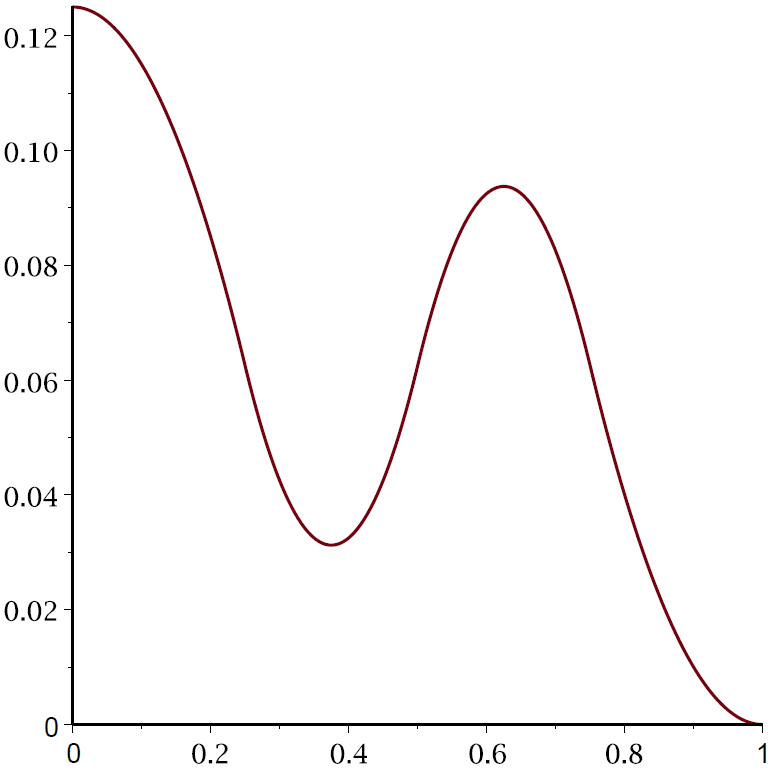}
         \caption{$\mu_2^+$}
             \end{subfigure}
     \hfill
     \begin{subfigure}[b]{0.3\textwidth}
         \centering
\includegraphics[width=\textwidth,height=0.6\textwidth]{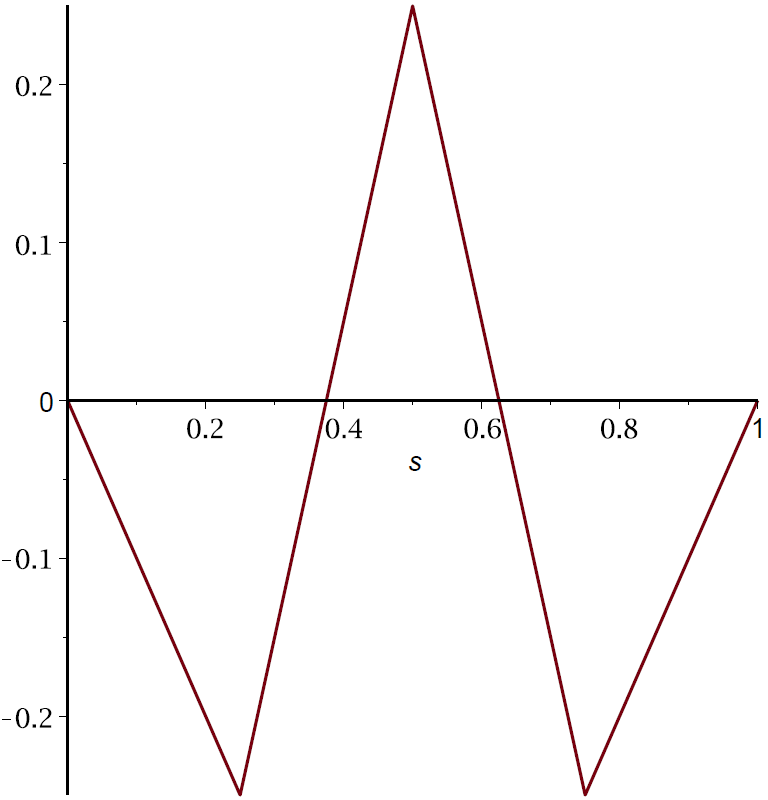}
         \caption{$\mu_1^-$}
             \end{subfigure}
        \caption{The graphs of the functions $\mu_1^+,\mu_2^+,$ $\mu_1^-$ related to  ineq. \eqref{1-3}.}
        \end{figure}
This means that not all $2$-increasing functions satisfy \eqref{1-3}. 
On the other hand, the function $\mu_2^+$ is nonnegative on $[0,1]$ (see Figure 1 (B)). Therefore, every function belonging to $M_{2^+,3^+}(I)$ satisfies the inequality \eqref{1-3}.

Surprisingly this inequality is also satisfied for $f\in M_{2^+,3^-}(I).$
Indeed, using Theorem \ref{sufth} with $\lambda=1$, it is enough to observe that $\mu_0(1)=\mu_1(1)=0,\mu_2(1)=\frac18$ and $\mu_2^-$ is nonnegative on $[0,1]$ (see Figure 2 (A)).

We also have that $\mu_0(1/2)=\mu_1(1/2)=\mu_3(1/2)=0$ and $\mu_2(1/2)=1/8$ and that $\mu_3^-$ is nonpositive and $\mu_3^+$ is nonnegative on $[0,1]$ (see Figure 2 (B) and (C)).
\begin{figure}
     \centering
          \begin{subfigure}[b]{0.3\textwidth}
         \centering        \includegraphics[width=\textwidth, height=0.6\textwidth]{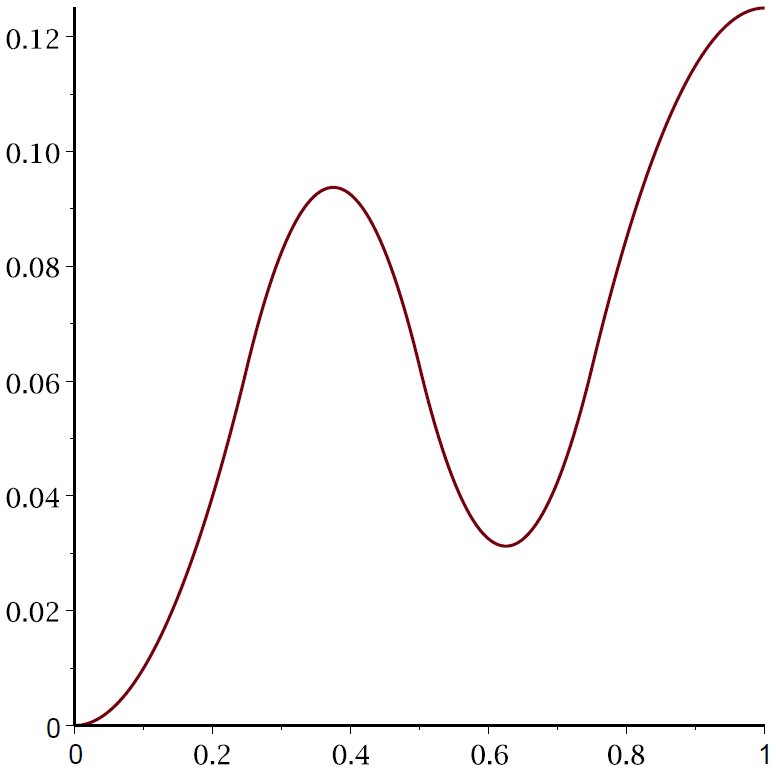}
         \caption{$\mu_2^-$}
             \end{subfigure}
     \hfill
       \begin{subfigure}[b]{0.3\textwidth}
         \centering
\includegraphics[width=\textwidth, height=0.6\textwidth]{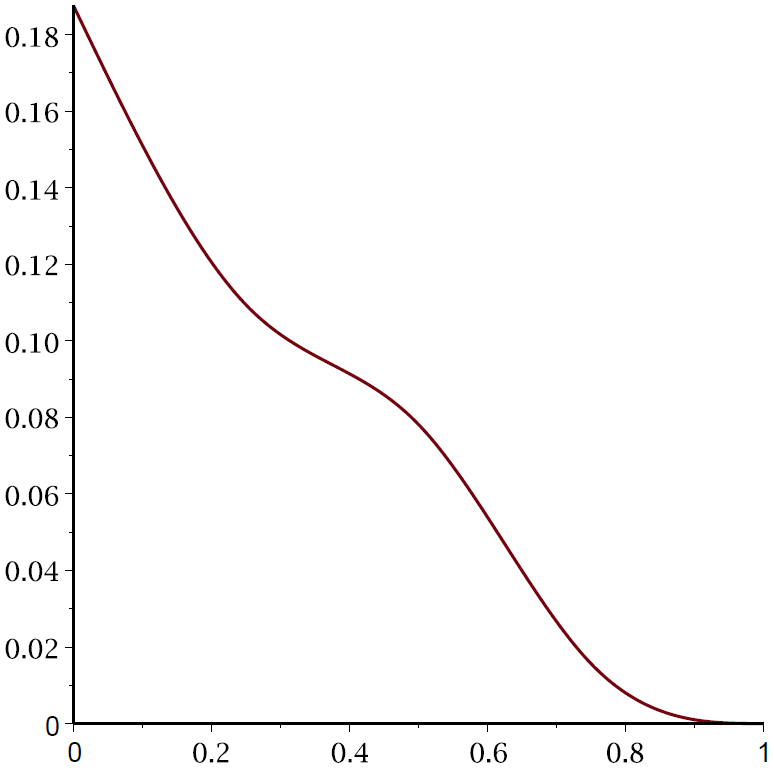}
         \caption{$\mu_3^+$}
             \end{subfigure}
     \hfill
     \begin{subfigure}[b]{0.3\textwidth}
         \centering
\includegraphics[width=\textwidth,height=0.6\textwidth]{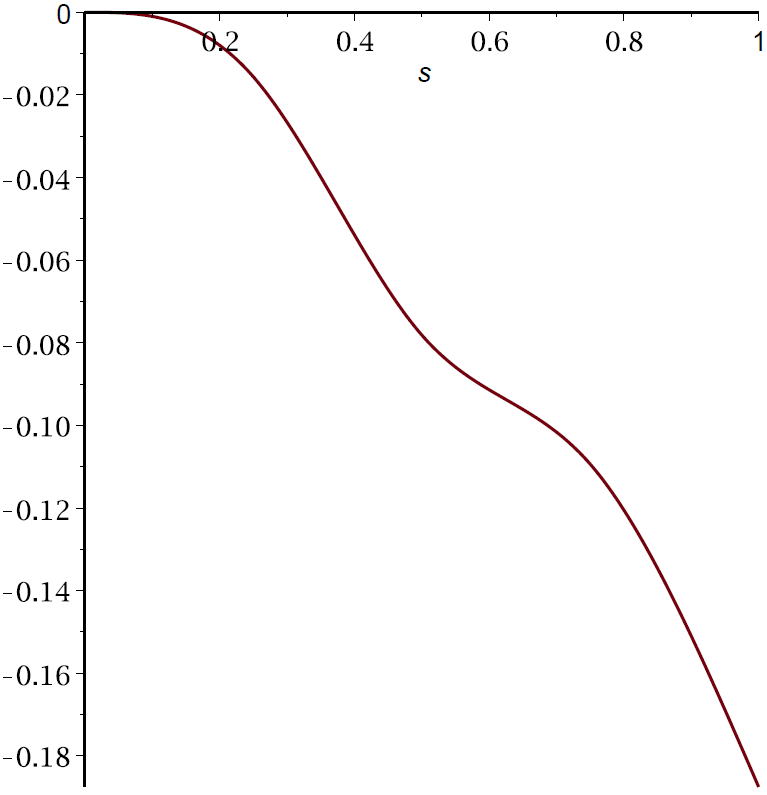}
         \caption{$\mu_3^-$}
             \end{subfigure}
        \caption{The graphs of the functions $\mu_2^-,\mu_3^-,$ $\mu_3^+$ related to  ineq. \eqref{1-3}.}
        \end{figure}

Thus, applying Theorem \ref{sufth} with $\lambda=1/2$, it follows that every function belonging to $M_{2^+,4^+}(I)$ also satisfies the inequality \eqref{1-3}.

Due to the linearity of the inequality \eqref{1-3}, we finally can obtain that every function $f$ belonging to the convex cone $$M_{2^+,3^+}(I)+M_{2^+,3^-}(I)+M_{2^+,4^+}(I)$$ is a solution to \eqref{1-3}.
\end{example}

In the next example, we present an inequality where three orders of increasingness must be used to ensure that a given function satisfies this inequality. 
\begin{example}
To deal with the inequality 
    \begin{equation}
\label{33222}
 3f\left(\frac{x+9y}{10}\right)+3f\left(\frac{9x+y}{10}\right)\leq 2f(x)+2f\left(\frac{x+y}{2}\right)+2f(y) 
 \end{equation}
we consider the measure
$$\mu=2\delta_0-3\delta_{\frac{1}{10}}+2\delta_\frac12-3\delta_{\frac{9}{10}}+2\delta_1.$$
We have $\mu_0(0)=\mu_1(0)=0$ and $\mu_2(0)=\frac{1}{25}.$
As in the previous example, every solution of \eqref{33222} must be $2$-increasing. However, this time, both functions $\mu_1^+$ and $\mu_2^+$ change their signs in $[0,1]$ (see Figure 3 (A) and (B)).
\begin{figure}
     \centering
          \begin{subfigure}[b]{0.3\textwidth}
         \centering        \includegraphics[width=\textwidth, height=0.6\textwidth]{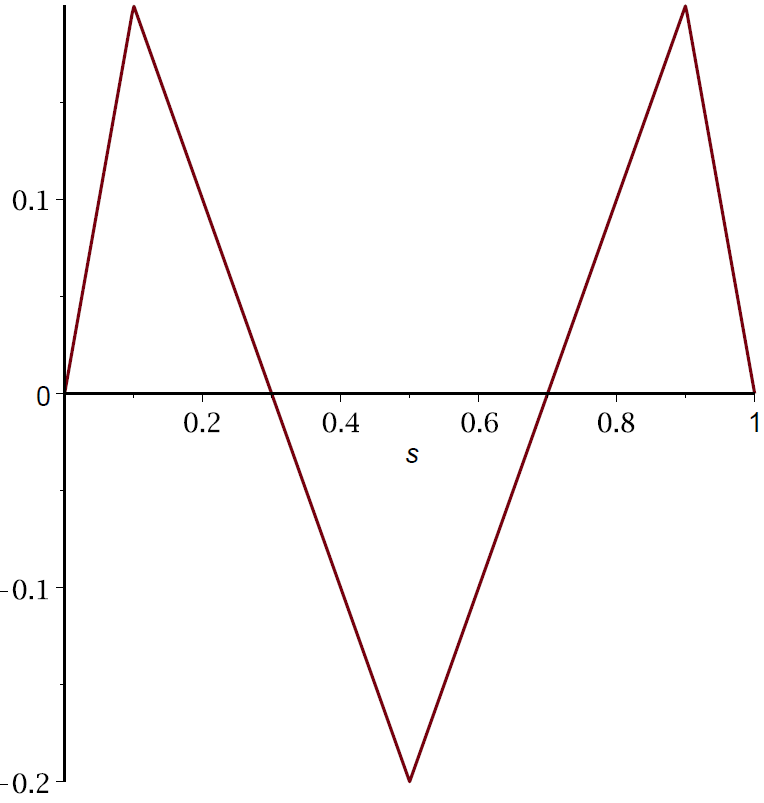}
         \caption{$\mu_1^+$}
         \label{fig:three sin x}
     \end{subfigure}
     \hfill
       \begin{subfigure}[b]{0.3\textwidth}
         \centering
\includegraphics[width=\textwidth, height=0.6\textwidth]{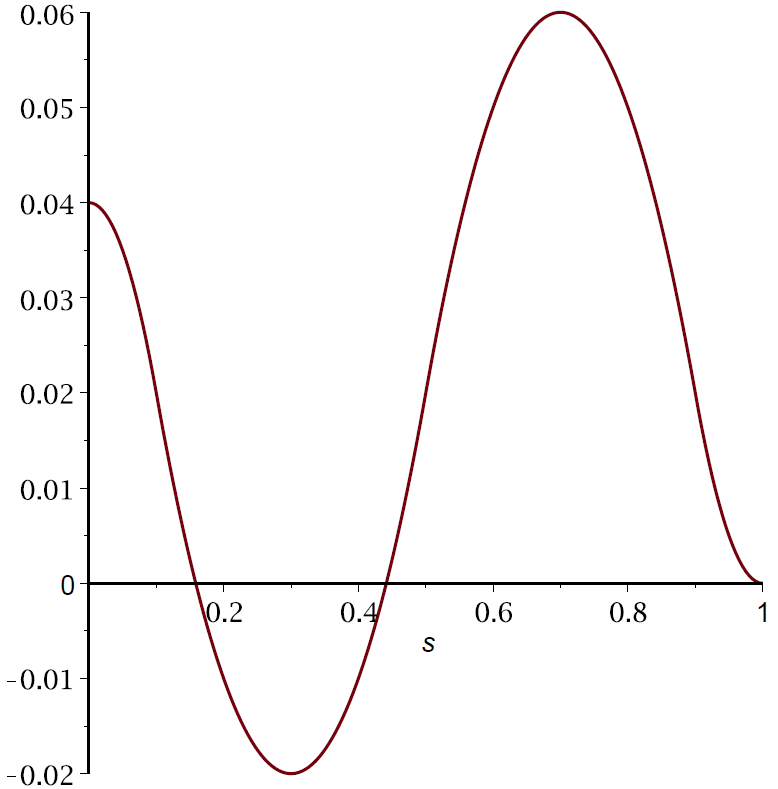}
         \caption{$\mu_2^+$}
             \end{subfigure}
     \hfill
     \begin{subfigure}[b]{0.3\textwidth}
         \centering
\includegraphics[width=\textwidth,height=0.6\textwidth]{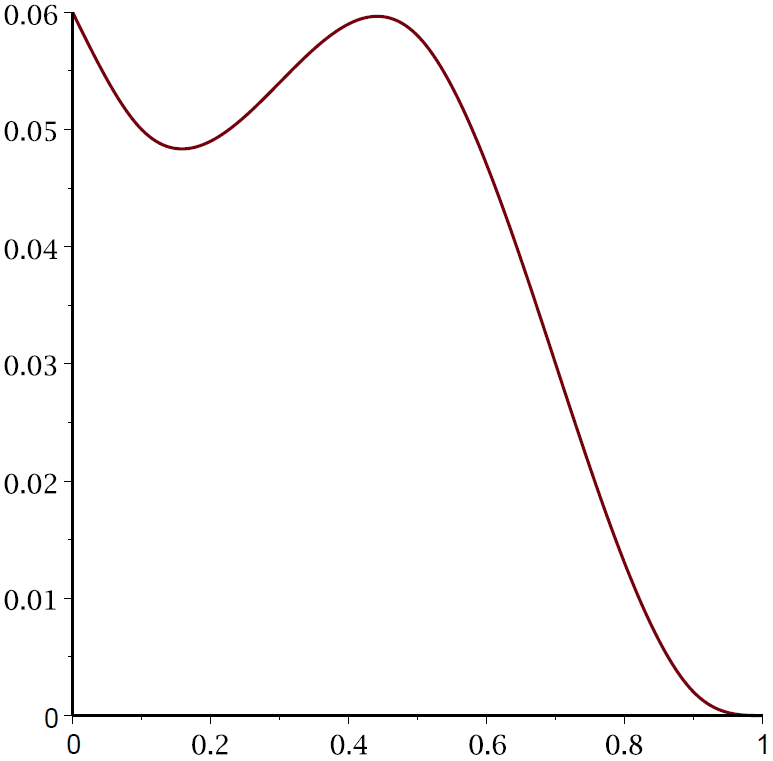}
         \caption{$\mu_3^+$}
             \end{subfigure}
        \caption{The graphs of the functions $\mu_1^+,\mu_2^+$ and $\mu_3^+$ related to the inequality \eqref{33222}.}
        \end{figure}
Therefore, we examine  $\mu_3^+$, and it turns out that it is positive (see Figure 3 (C)). This together with the equality $\mu_3^+(0)=\frac{3}{50}>0$ means that \eqref{33222} is satisfied by all functions $f$ from $M_{2^+,3^+,4^+}(I).$
\end{example}

\begin{remark}
In view of Lemma \ref{Lem:mk}, it is clear that if $\mu_k^+$ is nonnegative for some $k_0$ and $\lambda\in[0,1]$ on the interval $[\lambda,1]$, then it will be nonnegative for all $k\geq k_0$ on the interval $[\lambda,1]$. This will, however, bring no improvement since all classes of functions obtained in this way will be contained in the previously received. Thus, the procedure is to find the smallest $k$ for which $\mu_k^+$ does not change its sign on the interval $[\lambda,1]$.  

\end{remark}
In the previous two examples, the inequalities were symmetric, resulting in relatively symmetric classes of solutions.
Now, we present an example of a nonsymmetric inequality that will be satisfied by all functions from $M_{1^-,2^+}(I)$ and from $M_{1^-,3^+}(I)$. 

\begin{example}
Consider the inequality
\begin{equation}
\label{981}
9f\left(\frac{2x+y}{3}\right)\leq 8f(x)+f(y)    
\end{equation}
and the associated measure
$$\mu=8\delta_0-9\delta_\frac13+\delta_1.$$
We have $\mu_0(0)=0,$ $\mu_1(0)=-2$ which implies that all solutions of \eqref{981} are decreasing. The function $\mu_0^+$ changes its sign, meaning that not all decreasing functions satisfy \eqref{981}. 
\begin{figure}
     \centering
          \begin{subfigure}[b]{0.3\textwidth}
         \centering
         \includegraphics[width=\textwidth]{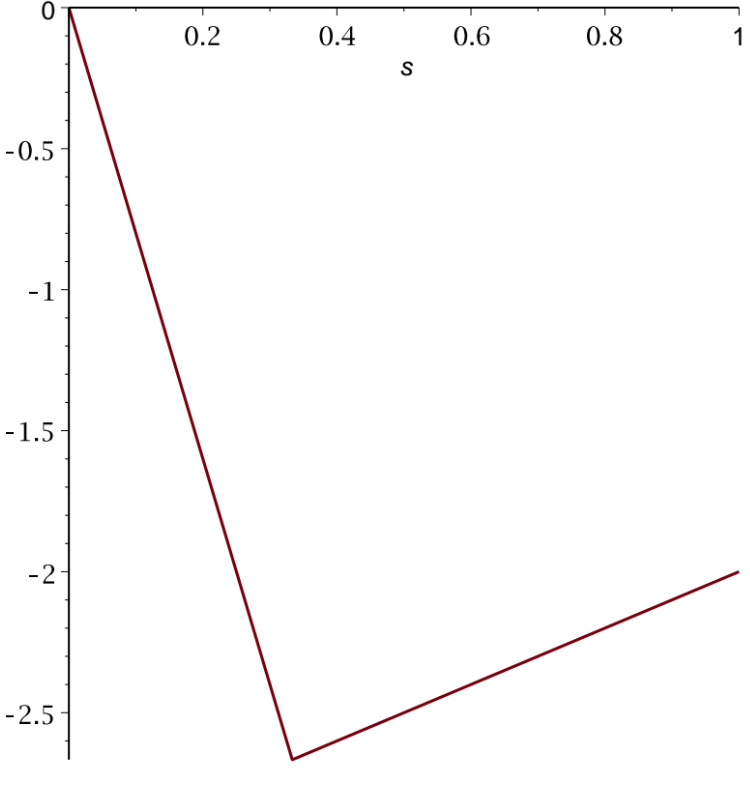}
         \caption{$\mu_1^-$}
              \end{subfigure}
     \hfill
          \begin{subfigure}[b]{0.3\textwidth}
         \centering
         \includegraphics[width=\textwidth]{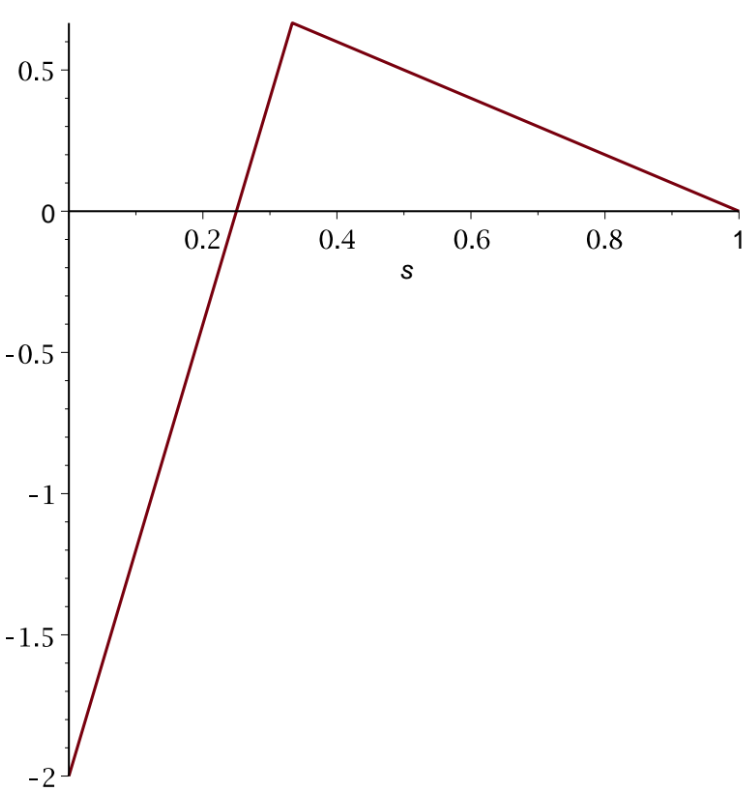}
         \caption{$\mu_1^+$}
              \end{subfigure}
     \hfill
     \begin{subfigure}[b]{0.3\textwidth}
         \centering
         \includegraphics[width=\textwidth]{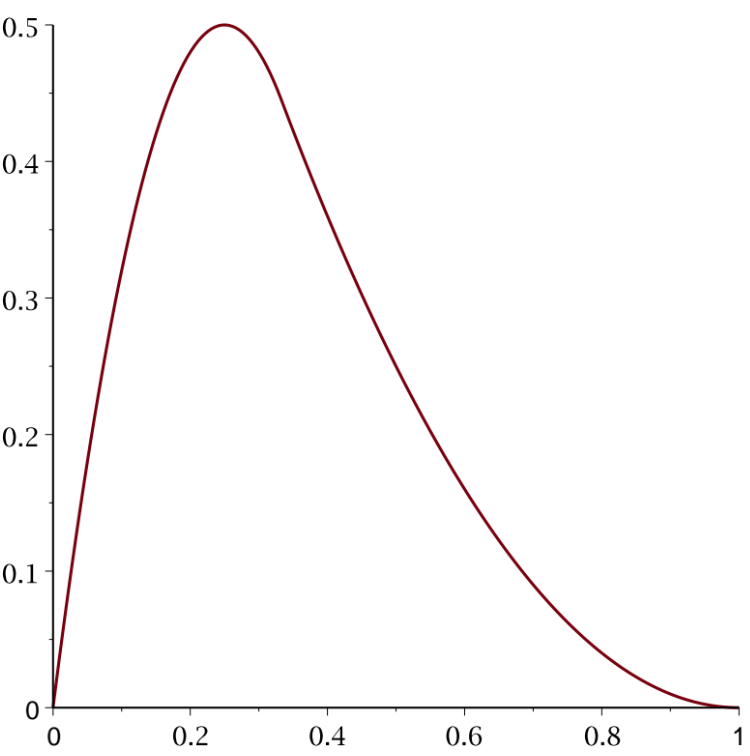}
         \caption{$\mu_2^+$}
              \end{subfigure}
        \caption{The graphs of $\mu_1^+,\mu_2^+,\mu_1^-$ related to ineq.\ \eqref{981}.}
        \end{figure}
Further, the function $\mu_1^+$ changes sign in $[0,1]$ but $\mu_1^-$ is nonpositive.  (See Figure 4(A) and (B).) Moreover, we have $\mu_0(1)=0,\mu_1(1)=-2$, consequently, applying Theorem~\ref{sufth} with $\lambda=1$, we can see that every function from $M_{1^-,2^+}(I)$ satisfies \eqref{981}.
On the other hand, the function $\mu_2^+$ is nonnegative and $\mu_2(0)=0$.
Therefore, applying Theorem~\ref{sufth} with $\lambda=0$, we obtain that every function from $M_{1^-,3^+}(I)$ also satisfies \eqref{981}. In conclusion, by the linearity of the functional inequality, it follows that every function from the convex cone $$M_{1^-,2^+}(I)+M_{1^-,3^+}(I)$$ fulfills \eqref{981}.
\end{example}

Observe that until now, the only class of solutions obtained with the use of $\lambda$ different from numbers $0$ or $1$ was $M_{2^+,4^+}$ from Example \ref{ex1-34-31}, where $\lambda=\frac12$ was used. However, even in that example, the functions $\mu_3^+,\mu_3^-$ were positive, resp. negative on the whole interval $[0,1].$ Observe that the assumption of Theorem \ref{sufth} is weaker. In the next example, we will use that theorem in its full strength. 
\begin{example}
Consider inequality
\begin{equation}
\label{927}
9f\left(\frac{2x+y}{3}\right)\leq 2f(x)+7f(y)    
\end{equation} 
\begin{figure}
     \centering
          \begin{subfigure}[b]{0.3\textwidth}
         \centering
         \includegraphics[width=\textwidth]{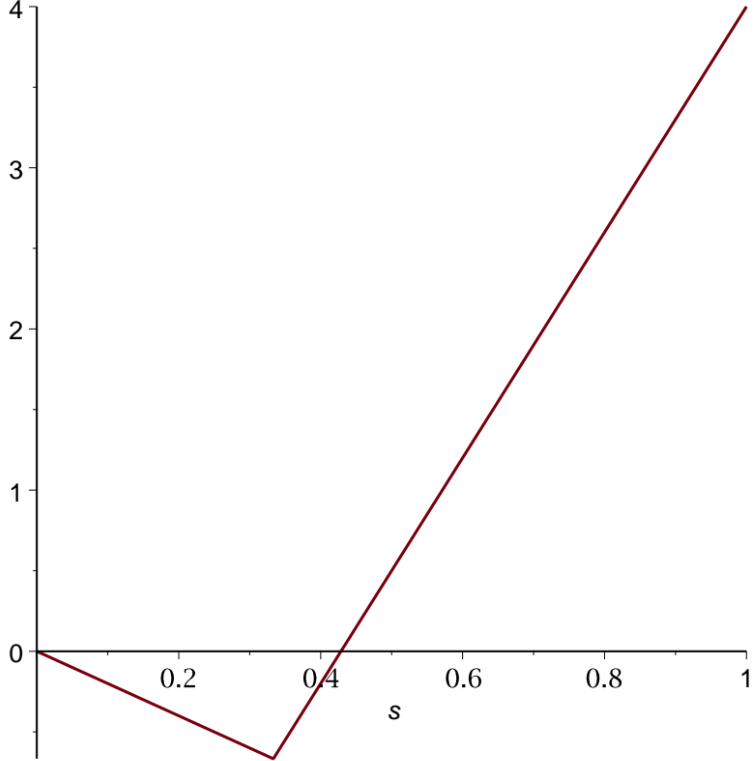}
         \caption{$\mu_1^+$}
              \end{subfigure}
     \hfill
          \begin{subfigure}[b]{0.3\textwidth}
         \centering
         \includegraphics[width=\textwidth]{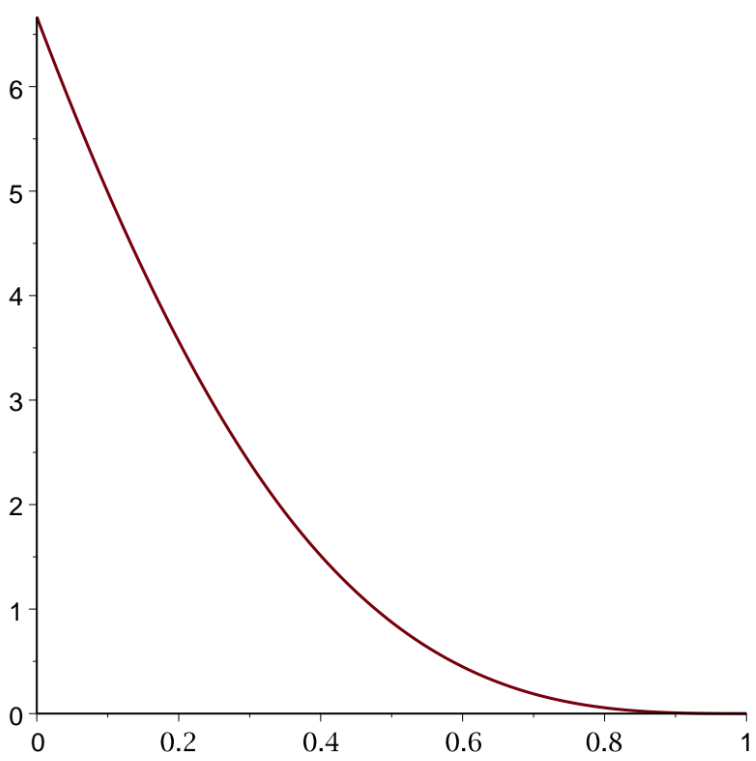}
         \caption{$\mu_2^+$}
              \end{subfigure}
     \hfill
     \begin{subfigure}[b]{0.3\textwidth}
         \centering
         \includegraphics[width=\textwidth]{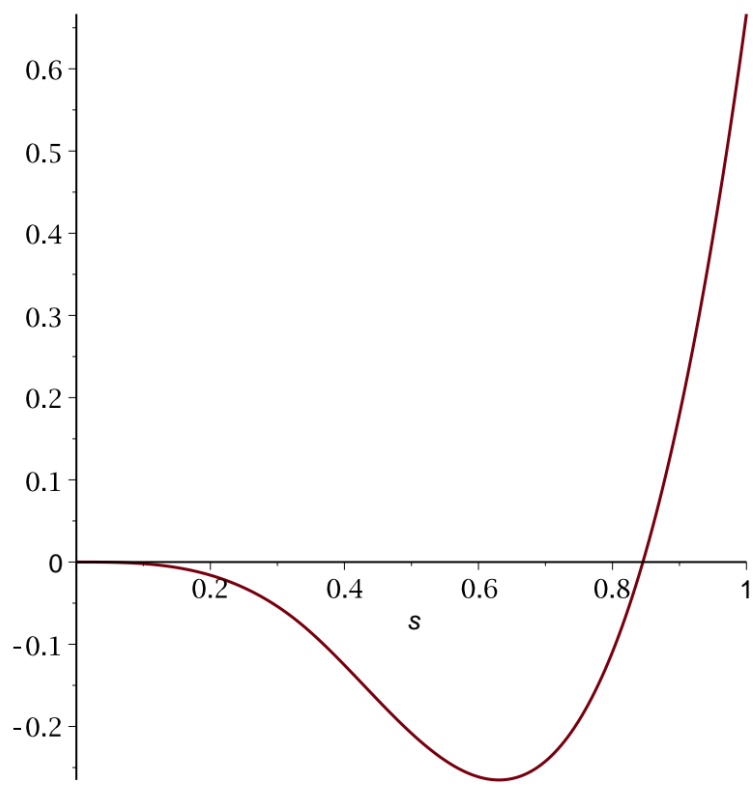}
         \caption{$\mu_2^-$}
              \end{subfigure}
        \caption{The graphs of the functions $\mu_1^+,\mu_2^+$ and $\mu_2^-$ connected with inequality \eqref{927}.}
        \end{figure}
and the measure 
$$\mu=2\delta_0-9\delta_\frac13+7\delta_1$$
Here we have $\mu_0(0)=0$ and $\mu_1(0)=4.$ This means that every solution of \eqref{927} is increasing. The positivity of $\mu_1^-$ implies that every function from $M_{1^+,2^-}$ satisfies \eqref{927}.

The function $\mu_2^-$, on the other hand, changes its sign on $[0,1]$, and therefore, we cannot use here Theorem \ref{sufth} with $\lambda=1.$ We will use here $\lambda=\frac34$.  We have $\mu_0(\frac34)=0,\mu_1(\frac34)=4,\mu_2(\frac34)=0.$ Further, we have 
\[
\mu_{2}^-(\tau)\leq0\mbox{ for }\tau\in[0,3/4)
\quad\mbox{and}\quad
\mu_{2}^+(\tau)\geq0\mbox{ for }\tau\in(3/4,1],
\]
see Figure 5 (B) and (C). Therefore, using Theorem \ref{sufth}, we get that \eqref{927} is also satisfied by all $f\in M_{1^+,3^+}.$ In conclusion, by the linearity of the functional inequality, it follows that every function from the convex cone $$M_{1^+,2^+}(I)+M_{1^+,3^+}(I)$$ fulfills \eqref{927}.
\end{example}

\begin{remark}
As we can see, if we are given a functional inequality of the form 
\begin{equation}
\label{linineq}
a_1f\left(\alpha_1x+(1-\alpha_1)y\right)+\dots+
a_nf\left(\alpha_nx+(1-\alpha_n)y\right)\geq 0
\end{equation}
then, using Theorem \ref{sufth}, we can find a class of functions that satisfy this inequality. Interestingly, if the classical condition needed for an $n$-increasing ordering is satisfied, we get only one class of solutions. Still, if this is not the case, then different classes of functions may satisfy that inequality. 
\end{remark}

Given the above remark, we may say that the results obtained in this paper yield a step toward a complete solution of the inequality \eqref{linineq}. However, the applications of our results are not limited to dealing with inequalities of that kind. 
For example, it is well known that the Bullen inequality 
$$\frac{1}{y-x}\int_x^yf(t)dt\leq\frac{1}{4}f(x)+\frac12f\left(\frac{x+y}{2}\right)+\frac14f(y)$$
holds for all $2$-increasing (i.e., convex) functions $f$ and that 
$$\frac{1}{y-x}\int_x^yf(t)dt\leq\frac{1}{6}f(x)+\frac23f\left(\frac{x+y}{2}\right)+\frac16f(y)$$
is satisfied by all $4$-increasing functions (cf. \cite{BPDebr}).
The right-hand side of the above inequalities is of the form 
$$af(x)+(1-2a)f\left(\frac{x+y}{2}\right)+af(y),$$
with $a=\frac14,\frac16.$ It is natural to ask what happens for other values of $a.$ We will consider such a situation in the following example.
\begin{example}
The inequality 
\begin{equation}
\label{15}
\frac{1}{y-x}\int_x^yf(t)dt\leq\frac{1}{5}f(x)+\frac35f\left(\frac{x+y}{2}\right)+\frac15f(y)
\end{equation}
is satisfied by all 
\begin{equation}
\label{16}
f\in M_{2^+,3^+}(I)+M_{2^+,3^-}(I)+M_{2^+,4^+}(I)
\end{equation}
for all $x,y\in I$ with $x<y$.
Indeed, consider the measure
$$\mu=\frac15\delta_0+\frac35\delta_\frac12+\frac15\delta_1-\ell,$$
where $\ell$ is the Lebesgue measure restricted to the measurable subsets of $[0,1].$

Then we have $\mu_0(0)=\mu_1(0)=0$ and $\mu_2(0)=\frac{1}{60}.$ 
\begin{figure}
     \centering
          \begin{subfigure}[b]{0.3\textwidth}
         \centering
         \includegraphics[width=\textwidth]{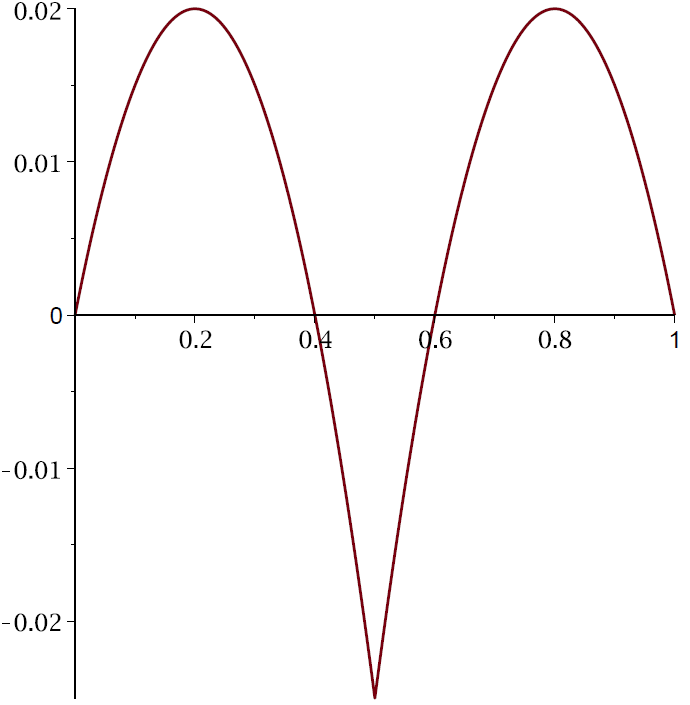}
         \caption{$\mu_1^+$}
              \end{subfigure}
     \hfill
          \begin{subfigure}[b]{0.3\textwidth}
         \centering
         \includegraphics[width=\textwidth]{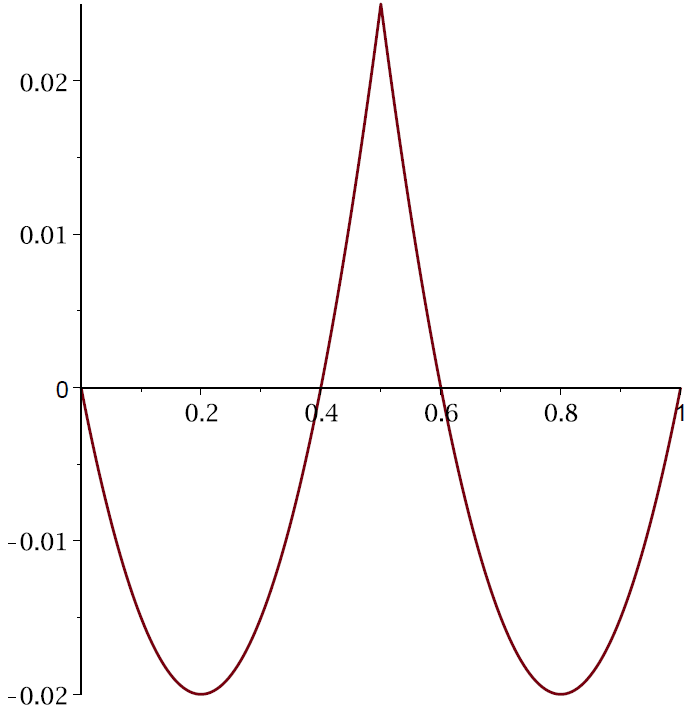}
         \caption{$\mu_1^-$}
              \end{subfigure}
     \hfill
     \begin{subfigure}[b]{0.3\textwidth}
         \centering
         \includegraphics[width=\textwidth]{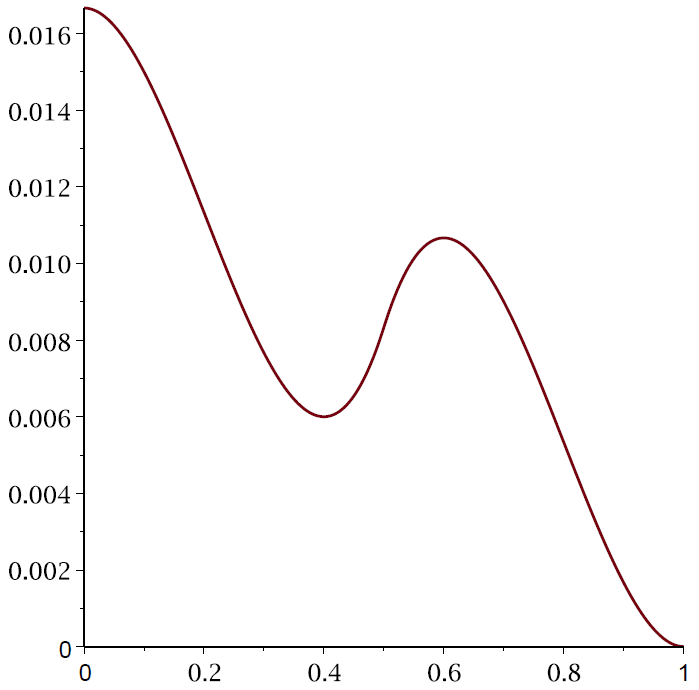}
         \caption{$\mu_2^+$}
              \end{subfigure}
        \caption{The graphs of the functions $\mu_1^+,\mu_1^-$ and $\mu_2^+$ connected with inequality \eqref{15}.}
        \end{figure}

Further, the function $\mu_1^+$ and $\mu_1^-$ change sign in $[0,1]$ (see Figure 6 (A) and (B)) therefore the inequality is not satisfied by all 2-increasing functions. On the other hand, $\mu_2^+$ and $\mu_2^-$ are nonnegative on $[0,1]$ (see Figure 6 (C) and Figure 7 (A).
\begin{figure}
     \centering
          \begin{subfigure}[b]{0.3\textwidth}
         \centering
         \includegraphics[width=\textwidth]{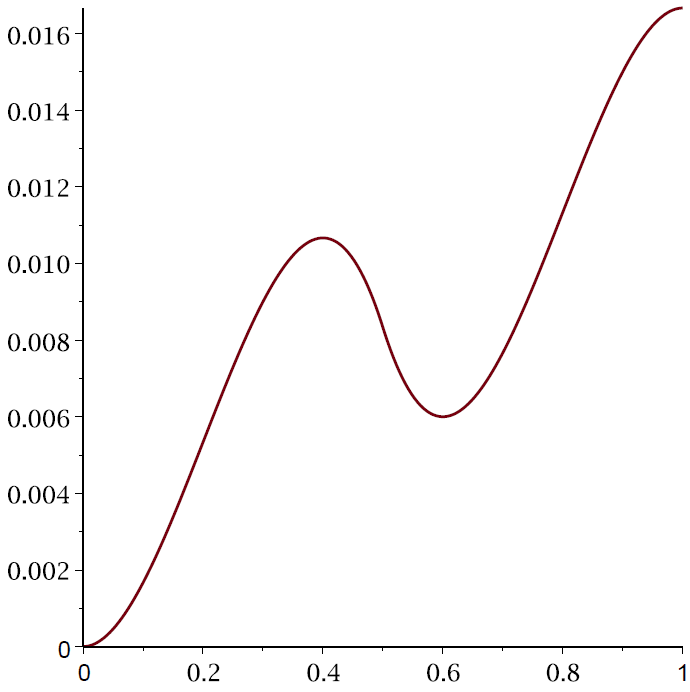}
         \caption{$\mu_2^-$}
              \end{subfigure}
     \hfill
          \begin{subfigure}[b]{0.3\textwidth}
         \centering
         \includegraphics[width=\textwidth]{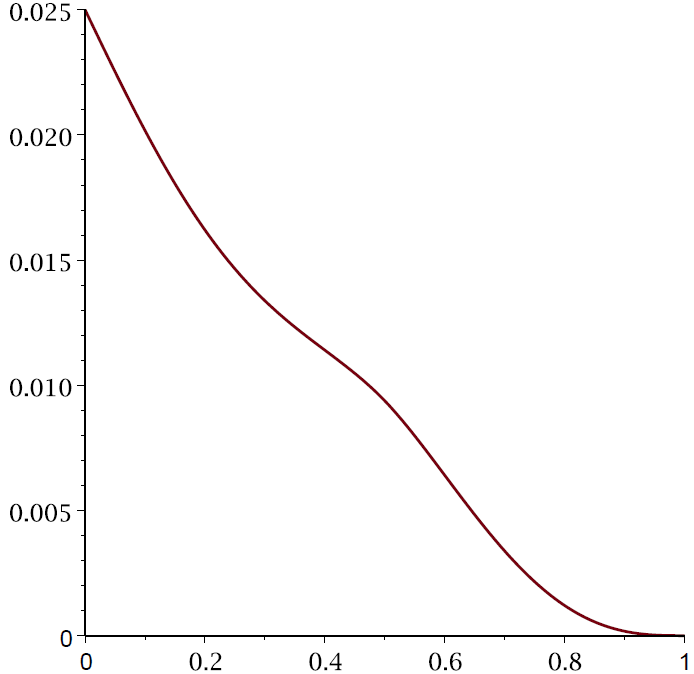}
         \caption{$\mu_3^+$}
              \end{subfigure}
     \hfill
     \begin{subfigure}[b]{0.3\textwidth}
         \centering
         \includegraphics[width=\textwidth]{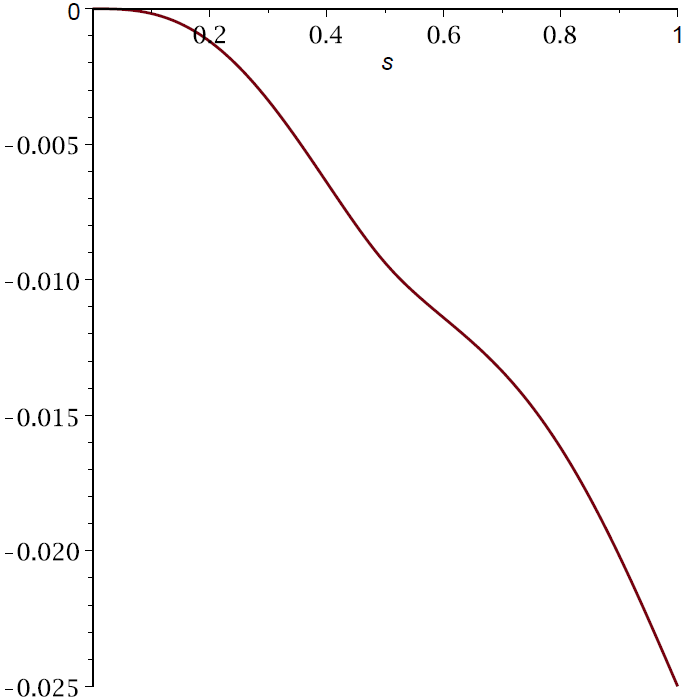}
         \caption{$\mu_3^-$}
              \end{subfigure}
        \caption{The graphs of the functions $\mu_2^-,\mu_3^+$ and $\mu_3^-$ connected with inequality \eqref{15}.}
        \end{figure}

Therefore, applying Theorem \ref{sufth} with $n=3$ and $\lambda\in\{0,1\}$, we obtain that \eqref{15} is satisfied by all $f\in M_{2^+,3^+}(I)$ and $f\in M_{2^+,3^-}(I)$, respectively. The function $\mu_3$ vanishes at $\lambda=1/2$ and condition (b)(i) of Theorem \ref{sufth} with $n=4$ and $\lambda=1/2$ is satisfied (see Figure 7 (B) and (C)), therefore, \eqref{15} also holds for all $f\in M_{2^+,4^+}(I)$. Using the linearity of \eqref{15}, it now follows that it holds for all $f$ satisfying \eqref{16}.
\end{example}

\end{document}